\documentclass[11pt, twoside, leqno]{article}
\usepackage{amsmath,amsthm}
\usepackage{amssymb,latexsym,bm,mathrsfs}
\usepackage{enumerate}
\usepackage[all]{xy}
\usepackage{amscd,multicol,hhline}
\usepackage{hyperref}
\usepackage{graphicx,color}

\newtheorem{theorem}{Theorem}
\newtheorem{prop}[theorem]{Proposition}
\newtheorem{lemma}[theorem]{Lemma}

\newtheorem{claim}{Claim}

\newtheorem{fact}{Fact}

\newtheorem*{cor}{Corollary}

\theoremstyle{definition}
\newtheorem{definition}{Definition}

\newtheorem*{rema}{Remark}

\newtheorem*{defin}{Definition}
\newtheorem*{notat}{Notation}

\numberwithin{equation}{section}

\newcommand{\Z}{{\mathbb Z}}
\newcommand{\R}{{\mathbb R}}
\def\wt{\widetilde }

\newcommand{\im}{\mathop{\mathrm{im}}\nolimits}
\newcommand{\rk}{\operatorname{rank}}
\def\Cob{\text{\rm Cob}^{SO}}
\def\Coker{\text{\rm Coker}\,}
\newcommand{\Prim}{{\text{\rm Prim}^{SO}}}
\newcommand{\down}{{\bm{\downarrow}}}
\newcommand{\up}{{\bm{\uparrow}}}
\newcommand{\Diff}{{\operatorname{Diff}}}
\allowdisplaybreaks

\begin{document}

\renewcommand{\thefootnote}{\arabic{footnote}}
\setcounter{footnote}{0}

\title{Singularities and stable homotopy groups of spheres. I}

\author{by\\
\\
Csaba Nagy, Andr\'as Sz\H{u}cs and Tam\'as Terpai}

\date{}

\maketitle

%% diagnostic
%% \tableofcontents

\renewcommand{\thefootnote}{}
\footnotetext{{\it Acknowledgement:} The authors were supported by the National Research, Development and Innovation Office NKFIH (OTKA) Grants NK 112735 and K 120697, the second and third authors were partially supported by ERC Advanced Grant LDTBud.}

%%\footnote{2010 \emph{Mathematics Subject Classification}: Primary 11N05; Secondary 11N35, 11P32.}
%%\footnote{\emph{Key words and phrases}: Gaps between primes, Goldbach conjecture.}

\renewcommand{\thefootnote}{\arabic{footnote}}
\setcounter{footnote}{0}

\def\thesection{\S\arabic{section}}
%%%%%%%%

\begin{abstract}
We establish an interesting connection between Morin singularities and stable homotopy groups of spheres. We apply this connection to computations of cobordism groups of certain singular maps. The differentials of the spectral sequence computing these cobordism groups are given by the composition multiplication in the ring of stable homotopy groups of spheres.
\end{abstract}

%In \cite{Sz3} the cobordism groups of cusp maps of $n$-dimensional oriented manifolds into $\mathbb R^{n + 1}$ were computed modulo their $2$- and $3$-primary parts.
%Here we compute the $3$-torsion parts of these groups (up to some group extension).
%Further we describe some interesting elements of the stable homotopy groups of spheres arising from singularities.

\section{Introduction}
We are considering locally stable smooth maps of $n$-dimensional manifolds into $(n+1)$-dimensional manifolds with the simplest, corank $1$ singularities (those where the rank of the differential map is at most $1$ less than the maximum possible). These singularities, called Morin singularities, form a single infinite family, with members denoted by symbols $\Sigma^0$ (nonsingular points), $\Sigma^{1,0}$ (fold points), $\Sigma^{1,1,0}=\Sigma^{1_2}$ (cusp points), $\dots$, $\Sigma^{1,\underset{j}{\dots},1,0}=\Sigma^{1_j}$, $\dots$ (see \cite{M}). A map that only has singularities $\Sigma^{1_j}$ with $j\leq r$ is called a $\Sigma^{1_r}$-map, and we are interested in calculating the cobordism groups of such maps. Two $\Sigma^{1_r}$-maps with the same target manifold $P$ are ($\Sigma^{1_r}$-)cobordant if there exists a $\Sigma^{1_r}$-map into $P\times [0,1]$ from a manifold with boundary whose boundary is the disjoint union of those two maps. Unless specified otherwise maps between manifolds will be assumed to be cooriented (equipped with an orientation of the virtual normal bundle), Morin and of codimension $1$ (that is, the dimension of the target is $1$ greater than the dimension of the source).
\par

The cobordism group of fold maps of oriented $n$-manifolds into $\R^{n+1}$ -- denoted by $\Cob\Sigma^{1,0}(n+1)$ -- was computed in \cite{Sz3} completely, while that of cusp maps, $\Cob\Sigma^{1,1,0}(n+1)$, only modulo $2$- and $3$-torsion.\footnote{These groups were denoted in \cite{Sz3} by ${\rm Cob}\Sigma^{1,0}(n)$ and ${\rm Cob}\Sigma^{1,1}(n)$, respectively; note the shifted dimensions.}
%The cobordism group of cusp maps (that is, maps that have only at most fold and cusp singular points) of oriented $n$-manifolds into $\R^{n+1}$ -- denoted by $\Cob\Sigma^{1,1,0}(n+1)$ -- was computed modulo $2$- and $3$-primary torsion parts in \cite{Sz3}.
Here we compute the $3$-torsion part (up to a group extension).
We shall also consider a subclass of such maps, the so-called prim (projection of immersion) cusp maps. These are the cusp maps with trivial and trivialized kernel bundle of the differential over the set of singular points. The cobordism group of prim fold and cusp maps of oriented $n$-manifolds to $\R^{n+1}$ will be denoted by $\Prim\Sigma^{1,0}(n+1)$ and $\Prim\Sigma^{1,1,0}(n+1)$ respectively. We shall compute these groups up to a group extension, and in the case of cusp maps up to $2$-primary torsion.

\section{Notation and formulation of the result}\label{section:notation}

We shall denote by $\pi^s(n)$ the $n$th stem, that is,
$$
\pi^s(n) = \underset{q \to \infty}{\lim} \pi_{n+q}(S^q).
$$
Denote by $\mathcal G$ the direct sum $\bigoplus\limits_{n} \pi^s(n)$. Recall that $\mathcal G$ is a ring, with multiplication $\circ$ defined by the composition (see \cite{To}). This product is skew commutative: for homogeneous elements $x$ and $y$ of $\mathcal G$
$$
x \circ y = (-1)^{\dim x \cdot \dim y} y \circ x.
$$
Recall that $\pi^s(3) \cong \Z_3 \oplus \Z_8$; the standard notation for the generator of the $\Z_3$ part is $\alpha_1$. By a slight abuse of notation we shall also denote by $\alpha_1$ the group homomorphism of $\mathcal G$ defined by left multiplication by $\alpha_1$, i.e. $\alpha_1(g) =\alpha_1 \circ g$ for any $g \in \mathcal G$.
\par
To formulate our results more compactly, we use the language of Serre classes of abelian groups \cite{S}. In particular, we will denote by $\mathcal C_2$ the class of finite $2$-torsion groups and $\mathcal C_{\{2,3\}}$ will denote the class of finite groups of order a product of powers of $2$ and $3$. Given a Serre class $\mathfrak C$, we call a homomorphism $f:A\to B$ a $\mathfrak C$-isomorphism if both $\ker f \in \mathfrak C$ and $\Coker f \in \mathfrak C$. Two groups are considered to be isomorphic modulo $\mathfrak C$ (denoted by $\underset{\mathfrak C}{\cong}$) if there exists a chain of $\mathfrak C$-isomorphisms that connects them. For example, isomorphism modulo $\mathcal C_{\{2,3\}}$ is isomorphism modulo the $2$-primary and $3$-primary torsions. A complex $\dots \overset{d_{m+1}}{\to} C_m \overset{d_m}{\to} \dots$ of maps is called $\mathfrak C$-exact if both $\im (d_m \circ d_{m+1})$ and $\ker d_m/(\ker d_m \cap \im d_{m+1})$ belong to $\mathfrak C$ for all $m$.

\begin{theorem}\label{th:1}
There is a $\mathcal C_2$-exact sequence
$$
\aligned
0 \to \Coker\bigl(\alpha_1 : \pi^s(n - 3) &\to \pi^s(n)\bigr) \to \Cob \Sigma^{1,1,0}(n+1) \to\\
 &\to\ker\bigl(\alpha_1 : \pi^s(n - 4) \to \pi^s(n - 1)\bigr) \to 0.
 \endaligned
%% 0 \to \Coker (\alpha_1: \pi^s(n-3) \to \pi^s(n)) \to \Cob\Sigma^{1,1,0}(n+1) \to \ker (\alpha_1: \pi^s(n-4) \to \pi^s(n-1)) \to 0.
$$
\end{theorem}

\begin{rema}
Recall that in \cite{Sz3} it was shown that
$$
\Cob \Sigma^{1,1,0}(n+1) \underset{\mathcal C_{\{2,3\}}}{\cong} \pi^s(n) \oplus \pi^s(n - 4)
$$
(in particular $\Cob\Sigma^{1,1,0}(n+1)$ is finite unless $n=0$ or $n=4$, when its rank is $1$). Since $\alpha_1$ is a homomorphism of order $3$, Theorem \ref{th:1} is compatible with this result and determines the $3$-primary part of $\Cob \Sigma^{1,1}(n+1)$ up to a group extension.
\end{rema}

\begin{rema}
Recall (\cite[Chapter XIV]{To}) the first few groups $\pi^s(n)$ and the standard names of the generators of the $3$-primary components (with the exception of $\alpha_3$, which is denoted by $\alpha'_3$ in \cite{To}):
\par
\hskip-1.4cm
\begin{tabular}{l|c|c|c|c|c|c|c|c|c|c|c|c}
$n$&0&1&2&3&4&5&6&7&8&9&10&11\\
\hline
$\pi^s(n)$& $\Z$ & $\Z_2$ & $\Z_2$ & $\Z_{24}$ & $0$ & $0$ & $\Z_2$ & $\Z_{240}$ & $\Z_2^2$ & $\Z_2^3$ & $\Z_6$ & $\Z_{56} \times \Z_9$ \\
\hline
$\left(\pi^s(n)\right)_3$& $\Z$ & & & $\Z_3\langle \alpha_1 \rangle$ & & & & $\Z_3\langle \alpha_2 \rangle$ & & & $\Z_3 \langle \beta_1 \rangle$ & $\Z_9\langle \alpha_3 \rangle$\\
\end{tabular}
Here and later we denote by $(G)_3$ the $3$-primary part\footnote{the quotient by the subgroup of torsion elements with orders coprime to $3$} of the abelian group $G$, omitting the parentheses when this does not cause confusion.
\par
For $n\leq 11$ the only occasion when the homomorphism $\alpha_1$ is non-trivial on the $3$-primary part is the epimorphism $\pi^s(0) \to \pi^s(3)$ (the homomorphism $\pi^s(7) \to \pi^s(10)$ is trivial modulo $3$ as a consequence of \cite[Lemma 13.8.]{To} and the fact that $(\pi^s(10))_3 \cong \Z_3$). Hence $\left(\Cob\Sigma^{1,1,0}(n+1)\right)_3$ fits into the short exact sequence $0 \to \left(\pi^s(n)\right)_3 \to \left(\Cob\Sigma^{1,1,0}(n+1)\right)_3 \to \left(\pi^s(n-4)\right)_3 \to 0$ for $n\leq 11$, $n\neq 3$.
\end{rema}

\section{Elements of stable homotopy groups of spheres arising from singularities}\label{section:singular}

\subsection{Representing elements of $\mathcal G$}

The following is a well-known corollary of the combination of the Pontryagin-Thom isomorphism and the Smale-Hirsch immersion theory:

\begin{fact}\label{fact:ImmFr}
The cobordism group of framed immersions of $n$-manifolds into $\R^{n+k}$ is isomorphic to $\pi^s(n)$ for any $k \geq 1$.
\end{fact}
\noindent
An equivalent formulation is the following:
\par
\begin{fact}\label{fact:normalFr}
The cobordism group of pairs $(M^n,F)$, where $M^n$ is a stably parallelizable $n$-manifold and $F$ is a trivialization of its stable normal bundle is isomorphic to $\pi^s(n)$.
\end{fact}

\par

These Facts follow from the so-called (Multi-)Compression Theorem (Theorem \ref{thm:multicompression} below).

\begin{definition}
Let $K$ be a compact subcomplex of a smooth manifold $M$. We say that $M$ \emph{retracts nicely} onto $K$ if there is a homotopy $\rho_t$, $t\in [0,1]$ (that we call a \emph{nice retraction} of $M$ onto $K$) such that
\begin{itemize}
\item $\rho_0=id_M$,
\item $\rho_t|_K=id_K$ for all $t\in [0,1]$,
\item every $\rho_t$ for $t<1$ is an embedding of $M$ into itself,
\item the image of $\rho_1$ is $K$.
\end{itemize}
\end{definition}

Before formulating the theorem, note that given natural numbers $k$, $m$ and $n \geq m$ as well as an immersion $i: M^m \looparrowright \R^n \times \R^k$ of a compact $m$-dimensional manifold $M^m$ with $k$ independent normal vector fields ${v_1,\ \dots,\ v_k}$, one can consider an extension of $i$ to $M \times D^k_\varepsilon$ for a small $k$-disc $D^k_\varepsilon = \{ (y_1, \dots, y_k) : \sum y_j^2 < \varepsilon\}$ for any sufficiently small positive $\varepsilon$:
\begin{align*}
\hat i : M \times D^k_\varepsilon &\looparrowright \R^n \times \R^k \\
(p,y_1, \dots,y_k) &\mapsto i(p)+\sum_{j=1}^k y_j v_j(p)
\end{align*}

\begin{theorem}[\cite{compression}]\label{thm:multicompression}
\begin{enumerate}[a)]
\item Given an embedding $i:M^m \hookrightarrow \R^n \times \R^k$ of a compact manifold $M$ (possibly with boundary) equipped with $k$ linearly independent normal vector fields $v_1$, $\dots$, $v_k$, such that $n>m$, there is an isotopy $\Phi_t$ of $\R^n \times \R^k$ such that $\Phi_0$ is the identity and $d\Phi_1(v_j)$ is parallel to the $j$th coordinate vector $\mathbf e_j$ of $\R^k$ for all $j$ with $1 \leq j \leq k$.

The same statement remains true if $n=m$ and each component of $M$ is a compact manifold with non-empty boundary.

\item If the vector fields $v_j$, $j\in I \subseteq \{1,\dots,k\}$ were already parallel to the corresponding vectors $\mathbf e_j$, $j\in I$, the vector fields $v_{\hat\jmath}$, $\hat \jmath \not\in I$ were orthogonal to all $\mathbf e_j$, $j\in I$, and either $n>m$ or $n=m$ and there exists a compact $(n-1)$-dimensional cell complex $K\subset M$ onto which $M$ retracts nicely via a homotopy $\rho_t$ for which the composition homotopy $i\circ\rho_t$ keeps the coordinates of $\R^k$ that belong to $I$ fixed, then for any $\hat\jmath \not \in I$ the isotopy $\Phi_t$ can be chosen so that
\begin{itemize}
\item the coordinates of $\R^k$ that belong to $I$ are kept fixed;
\item the map $(p,x,t) \mapsto \hat \Phi_t(p,x) := \Phi_t(\hat i(p,x))$ with $p\in M$, $x\in D^{\vert I \vert}_\varepsilon$ and $t \in [0, \infty)$ is an immersion of $M \times D^{\vert I\vert}_\varepsilon \times [0, \infty)$, where  $D^{\vert I \vert}_\varepsilon= \{(y_j)_{j\in I}: \sum y_j^2 < \varepsilon\}$;
\item for some time $t_0$ and all $t\geq t_0$ we have $\Phi_t(p)=\Phi_{t_0}(p)+(t-t_0) \cdot \mathbf e_{\hat\jmath}$.
\end{itemize}

\item The isotopy $\Phi_t$ can be chosen in such a way that for any point $p \in M$ and any time $t_0 \in [0,\infty)$ the tangent vector of the curve $t \mapsto \Phi_t(i(p))$ at the point $\Phi_{t_0}(i(p))$ will not be tangent to the manifold $\Phi_{t_0}(i(M))$. That is, the map $M \times [0,\infty) \to \R^n \times \R^k$ defined by the formula $(p,t) \mapsto \Phi_t(i(p))$ is an immersion.

\item {\bf (relative version)} Let $L$ be a compact subset of~$\R^n$ such that the vector fields $v_1$, $\dots$, $v_k$ are already parallel to $\mathbf e_1$, $\dots$, $\mathbf e_k$ in an open neighbourhood of~$i^{-1}(L\times\R^k)$; if $n=m$, assume additionally that there exists a compact $(n-1)$-dimensional cell complex $K\subset M$ such that $i^{-1}(L\times \R^k) \subset K$ and $M$ retracts nicely onto $K$. Then the isotopy $\Phi_t$ can be chosen to keep the vector fields $v_1$, $\dots$, $v_k$ on~$i^{-1}(L\times\R^k)$ parallel to $\mathbf e_1$, $\dots$, $\mathbf e_k$ at all times.

\end{enumerate}
\end{theorem}
\noindent
The part $c)$ of the theorem is not stated explicitly in \cite{compression}. In Appendix $3$ we give a proof of the theorem with emphasis on part $c)$ and some simplifications in the case $k=1$.

\par

We shall mostly use the following consequence of Hirsch-Smale theory:
%% Before formulating a corollary that we shall use, note that given an immersion $i: M^m \looparrowright \R^n \times \R^k$ with $k$ independent normal vector fields ${v_1,\ \dots,\ v_k}$, one can consider an extension of $i$ to $M \times D^k_\varepsilon$ for a small $k$-disc $D^k_\varepsilon = \{ (y_1, \dots, y_k) : \sum y_j^2 < \varepsilon\}$ for any sufficiently small positive $\varepsilon$:
%%\begin{align*}
%%\hat i : M \times D^k_\varepsilon &\looparrowright \R^n \times \R^k \\
%%(p,y_1, \dots,y_k) &\mapsto i(p)+\sum_{j=1}^k y_j v_j(p)
%%\end{align*}

\begin{theorem}\label{thm:immersionCompression}
Given an \emph{immersion} $i: M^m \looparrowright \R^n \times \R^k$ of a compact manifold $M$ with $k$ independent normal vector fields $v_1$, $\dots$, $v_k$ for which $m<n$, there is a regular homotopy $\Phi_t : M\times D^k_\varepsilon \looparrowright \R^n \times \R^k$ such that
\begin{itemize}
\item $\Phi_0=\hat i$;
\item $d\Phi_1|_{M\times \mathbf 0}(v_j)$ is parallel to the $j$th coordinate vector $\mathbf e_j$ of $\R^k$ for all $j$ with $1\leq j \leq k$;
\item the image of $\Phi_t$ on $M \times \mathbf 0$ is an immersion of the cylinder $M \times [0,1]$ into $\R^n \times \R^k$.
\end{itemize}
If $m=n$ and $M$ has no closed components, then for any $(m-1)$-dimensional compact subcomplex $K \subset M$ onto which $M$ retracts nicely there exists a neighbourhood $U$ of $K$ in $M$ for which there is a regular homotopy $\Phi_t : U \looparrowright \R^n \times \R^k$ with the same properties. 
\end{theorem}

\begin{proof}
Applying \cite[Multi-compression Theorem 4.5]{compression} (\cite[Addendum (v) to Multi-compression Theorem 4.5]{compression} when $n=m$)\footnote{Note that although the cited theorem is about embeddings, \cite[Addendum (vi) to Multi-compression Theorem 4.5]{compression} ensures that embeddings can be replaced by immersions.} gives us a regular homotopy $\wt \Phi_t^{M}:M \looparrowright \R^n \times \R^k$ that comes from an isotopy of $M$ within its induced neighbourhood pulled back by $i$ from $\R^n\times \R^k$; this defines a regular homotopy $\wt \Phi_t:M\times D^k_\varepsilon \looparrowright \R^n \times \R^k$ that satisfies the required properties except the last one. By appending collars if necessary -- moving the slice $\{p\}\times D^{k-1}_\varepsilon$ (the disk $D^{k-1}_\varepsilon$ being spanned by $v_2(p)$, $\dots$, $v_k(p)$) with constant velocity $v_1(p)$ before performing $\wt\Phi_t$, and with constant velocity $\mathbf e_1$ after $\wt\Phi_t$, the whole rescaled back to time in $[0,1]$ -- we may also assume that for $t \in [0,\varepsilon) \cup (1-\varepsilon,1]$ we have $\frac{\partial\wt\Phi_t}{\partial t} = v_1$; consequently the differential of $\wt\Phi_t$ on $M \times \mathbf 0 \times \left( \{0\} \cup \{1\} \right)$ is nondegenerate and the image of $\wt\Phi_t$ on $M$ restricted to these values of $t$ is an immersion of $M \times ([0,\varepsilon) \cup (1-\varepsilon,1])$, possibly for a smaller value of $\varepsilon$. We can now extend $\wt\Phi_t$ to a map
\begin{align*}
\mathbf{\wt\Phi} : M \times [0,1] \times D^{k-1}_\varepsilon & \to \R^n\times\R^k\\
\mathbf{\wt\Phi}(p,t_1,\dots,t_k) &= \wt\Phi_{t_1}(p)+(d\wt\Phi_{t_1})_p (t_2v_2(p)+\dots+t_kv_k(p))
\end{align*}
that is an immersion for $t_1 \in [0,\varepsilon) \cup (1-\varepsilon,1]$ and define the bundle map
\begin{align*}
D : T(M\times[0,1]\times D^{k-1}_\varepsilon) &\to T(\R^n \times \R^k)&\\
(w,s_1,s_2,\dots,s_k)_{(p,t_1,t_2,\dots,t_k)}& \mapsto (d\wt\Phi_{t_1})_p (w+s_1v_1(p)+\dots+s_kv_k(p)) &\\
& \qquad\text{taken at the basepoint }\mathbf{\wt\Phi}(p,t_1,\dots,t_k) &
\end{align*}
covering $\mathbf{\wt\Phi}$. This map is fibrewise injective since the linear span of $v_1$, $\dots$, $v_k$ is normal to $TM$. In addition, $D$ coincides with the derivative map of the immersion $\mathbf{\wt\Phi}|_{M \times ([0,\varepsilon) \cup (1-\varepsilon,1]) \times D^{k-1}_\varepsilon}$, so by applying \cite[Theorem 5.7]{Hirsch} with $K=M\times[0,1]\times D^{k-1}_\varepsilon$ and $L=M \times \left([0,\varepsilon) \cup (1-\varepsilon,1]\right) \times D^{k-1}_\varepsilon$ we obtain an immersion $\Phi_t$ of $M\times [0,1]\times D^{k-1}_\varepsilon$ into $\R^n\times \R^k$ that coincides with $\wt\Phi_t$ on $M$ for $t\in [0,\varepsilon) \cup (1-\varepsilon,1]$. Reinterpreting $\Phi_t$ as a regular homotopy of $\Phi_0 = \wt \Phi_0 = i$ equipped with $k$ normal vector fields finishes the proof.
\end{proof}

\subsection{Framed immersions on the boundary of a singularity}\label{subsection:boundary}

We outline a connection between singularities and stable homotopy groups of spheres. Namely, we define elements of $\mathcal G$ that describe incidences of the images of singularity strata.
\par
We use the notation $\partial f$ to denote the \emph{link} of a map $f: \R^a \to \R^b$, that is, the restriction of $f$ to the preimage $f^{-1}(S^{b-1})$ of some sufficiently small sphere $S^{b-1} \subset \R^b$ centered at the origin. Later, we denote the same way the restriction of a map $g:(M,\partial M) \to (P,\partial P)$ to the boundary: $\partial g = g|_{\partial M}: \partial M \to \partial P$. This notation is consistent with the notion of a link in the sense that the link of $f: \R^a\to\R^b$ is the map $\partial g$, where $g: D^a \to D^b$ is the restriction of $f$ to the preimage $D^a$ of a sufficiently small ball $D^b$ in the target.
\par
{\bf Example $1$:} Let us consider the Whitney umbrella map
\[
\sigma_1: \mathbb R^2 \to \mathbb R^3, \ \ \ \sigma_1(t,x) = (t, tx, x^2).
\]
(the normal form of an isolated $\Sigma^{1,0}$-point). The preimage of the unit $2$-sphere $S^2 \subset \mathbb R^3$ is a closed curve $\sigma_1^{-1}(S^2) \subset \R^2$. The restriction of $\sigma_1$ to this curve -- the link of the map $\sigma_1$ -- is an immersion. The image $\sigma_1(\sigma^{-1}_1(S^2))$ is an immersed curve in $S^2$ with a single double point. The orientations of $\R^2$ and $\R^3$ give a coorientation on this curve. Hence this curve can be equipped with a normal vector in $S^2$ and with an additional normal vector to $S^2$ in $\R^3$, resulting in an immersed framed curve in $\R^3$; this represents an element of $\pi^s(1)$ that we shall denote by $d^1\sigma_1$. It is easy to see that $d^1\sigma_1 \neq 0$ (since this immersed curve in $S^2$ has a single double point). Using the standard notation $\eta$ for the generator of $\pi^s(1) = \Z_2$, we get that $d^1\sigma_1=\eta$.

\par

{\bf Example $2$:} Let us consider the normal form of a map with an isolated cusp-point at the origin
\begin{align*}
\sigma_2 : \R^4 &\to \R^5 \\
(t_1,t_2,t_3,x) &\mapsto (t_1,t_2,t_3,z_1,z_2)\\
z_1 &= t_1 x + t_2 x^2 \\
z_2 &= t_3 x + x^3
\end{align*}
The link of this map is its restriction to $\sigma_2^{-1}(S^4)$, where $S^4$ is the unit sphere in $\R^5$. Note that the $3$-manifold $\sigma_2^{-1}(S^4)$, which we shall denote by $L^3$, is diffeomorphic to $S^3$. The link map $\partial \sigma_2 = \sigma_2|_{L^3} :L^3 \to S^4$ is a fold map, it has only $\Sigma^{1,0}$ singularities along a closed curve $\gamma$. The image of this curve $\gamma$ in $S^4$ has a canonical framing. Indeed, the map $\sigma_2$ can be lifted to an embedding $\hat \sigma_2 : \R^4 \hookrightarrow \R^6$, $\hat\sigma_2 (t_1,t_2,t_3,x) = \left(\sigma_2(t_1,t_2,t_3,x),x\right)$ such that the composition of $\hat \sigma_2$ with the projection $\R^6 \to \R^5$ is $\sigma_2$. Hence the two preimages of any double point of $\sigma_2$ near the singularity curve $\sigma_2(\gamma)$ have an ordering and so one gets two of the normal framing vectors on the singularity curve $\sigma_2(\gamma)$. In order to get the third framing vector we note that $\sigma_2(\gamma)$ is the boundary of the surface formed by the double points of $\partial \sigma_2$ in $S^4$. The inward-pointing normal vector along $\sigma_2(\gamma)$ of this surface will be the third framing vector. (In Appendix $2$ we shall describe the framing that arises naturally on the image of the top singularity stratum of a map obtained as a generic projection of an immersion.) The curve $\sigma_2(\gamma)$ with this framing represents an element in $\pi^s(1)$ that we denote by $d^1(\sigma_2)$. We shall show that $d^1(\sigma_2)=0$ (see Lemma \ref{lemma:d1cusp} and Appendix $1$).
\par
In the present situation we can construct one more element associated to $\sigma_2$ in a quotient group of $\mathcal G$, which we shall denote by $d^2(\sigma_2)$. We construct this element (after making some choices) in $\pi^s(3) \cong \Z_{24}$ but it will be well-defined only in the quotient group $\Z_{24}/\Z_2$. The definition of $d^2(\sigma_2)$ is the following. By a result of \cite{Sz2} (that we shall recall in \ref{section:SS}, see also \cite{Te}) a framed cobordism of the embedded singularity curve $\sigma_2(\gamma)$ can be extended to a cobordism of the link map $\partial \sigma_2:L^3 \to S^4$. In particular, since the (framed) curve $\sigma_2(\gamma)$ is (framed) null-cobordant, the link map $\partial \sigma_2$ is fold-cobordant to an (oriented) immersion. That is, there is a compact oriented $4$-manifold $W^4$ such that $\partial W^4 = L^3 \sqcup V^3$ and there is a $\Sigma^{1,0}$-map $C:(W^4,L^3,V^3) \to (S^4 \times [0,1], S^4 \times \{0\}, S^4 \times \{1\} )$ such that the restriction $C|_{L^3}: L^3 \to S^4 \times \{0\}$ is $\partial \sigma_2$ and the restriction $C|_{V^3}: V^3 \to S^4 \times \{1\}$ is an immersion, which we denote by $\partial' \sigma_2$. It represents an element in $\pi^s(3)$ and its image in $\pi^s(3)/\Z_2$ is independent of the choice of $C$ (as detailed in the description of $d^2$ in subsection \ref{section:geometric}). The obtained element in $\pi^s(3)/\Z_2$ is $d^2(\sigma_2)$.
\par
%% (in the proof of Claim \ref{claim:product}) 
For future reference we introduce the notation $\sigma_2^*$ for the map
$$
\sigma_2^* = \sigma_2 \cup C : D^4 \underset{L^3}{\cup} W^4 \to D^5_2 = D^5 \underset{S^4\times \{0 \}}{\cup} S^4 \times [0,1].
$$
This definition makes sense since the maps $C$ and $\sigma_2$ coincide on the common part $L^3$ of their source manifolds, and it is easy to see that the gluing can be performed to make $\sigma_2^*$ smooth. Note that $\partial \sigma_2^*$ is the immersion $V^3 \looparrowright S^4 \times \{ 1 \}$.

%% (see Figure \ref{fig:sigmaStar}).
%%\par
%%\begin{figure}
%%\centering
%%\resizebox{10cm}{!}{\input{sigmaStar.pdf_t}}
%%\includegraphics[width=7cm]{sigmaStar.pdf}
%%\caption{Glued-together map $\sigma_2^*$}
%%\label{fig:sigmaStar}
%%\end{figure}

\par
We shall show that the stable homotopy group elements $d^1(\sigma_1)$, $d^1(\sigma_2)$, the coset $d^2(\sigma_2)$ and other analogously defined objects can be computed from a spectral sequence associated to the classifying spaces of singularities ($d^1$ and $d^2$ are in fact differentials of this spectral sequence). Next we shall describe these classifying spaces and the spectral sequence, first for the simpler case of prim maps.

\section{Classifying spaces of cobordisms of singular maps}\label{section:singularClass}

\begin{defin}
A smooth map $f: M^n \to P^{n + k}$ is called a prim map
(prim stands for the abbreviation of \emph{pr}ojected \emph{im}mersion) if
\begin{enumerate}[1)]
\item it is the composition of an immersion $g : M  \looparrowright P \times \mathbb R^1$ and the projection $p_P: P \times \mathbb R^1 \to P$, and
\item an orientation is given on the kernels of the differential of~$f$ that agrees with the orientation pulled back from $\R^1$ via the composition $p_\R \circ g$, where $p_\R : P \times \R^1 \to \R^1$ is the projection.
\end{enumerate}
For maps between manifolds with boundary $f:(M,\partial M)\to (P,\partial P)$, we shall always require that they should be {\emph{regular}}, that is, $f^{-1}(\partial P) = \partial M$ and the map $f$ in a neighbourhood of $\partial M$ can be identified with the direct product $f|_{\partial M} \times id_{[0,\varepsilon)}$ for a suitable positive $\varepsilon$.

\end{defin}

\begin{rema}
Note that for a Morin map the kernels of $df$ form a line bundle $\ker df \to \Sigma(f)$, where $\Sigma(f)$ is the set of singular points of~$f$. The conditions $1)$ and $2)$ ensure that for a prim map this bundle is orientable (trivial) and an orientation (trivialization) is chosen (the same map $f$ with a different choice of orientation on the kernels is considered to be a different prim map). The converse also holds: if a Morin map $f:M^n \to P^{n+k}$ is equipped with a trivialization of its kernel bundle, then there exists a unique (up to regular homotopy) immersion $g : M \looparrowright P \times \R^1$ such that $f = p_P \circ g$, where $p_P: P\times \R^1 \to P$ is the projection and the trivialization of the kernel bundle is the same as the one defined by the projection to $\R^1$.
\end{rema}

\begin{notat}
The cobordism group of all prim maps of $n$-dimensional oriented manifolds into $\mathbb R^{n + 1}$ will be denoted by $\Prim(n+1)$.
The analogous cobordism group of prim maps having only (at most) $\Sigma^{1,0}$-singular points (i.e.\ both the maps and the cobordisms between them are prim fold maps) will be denoted by $\Prim\Sigma^{1,0}(n+1)$; that of prim cusp maps will be denoted by $\Prim\Sigma^{1,1,0}(n+1)$; and that of prim $\Sigma^{1_i}$-maps will be denoted by $\Prim \Sigma^{1_i}(n+1)$.
\end{notat}

One can define cobordism sets of prim $\Sigma^{1,0}$ and $\Sigma^{1,1,0}$ (cooriented) maps of $n$-manifolds in arbitrary fixed $(n + 1)$-dimensional manifold $P^{n + 1}$ (instead of $\mathbb R^{n + 1}$).
The obtained sets we denote by $\Prim\Sigma^{1,0}(P)$ and \linebreak $\Prim \Sigma^{1,1,0}(P)$, respectively.\footnote{These sets carry a natural abelian group structure, but we do not use this fact here. In \cite[Remark 8]{Sz2} the analogous statement for not necessarily prim maps with $k=1$ is shown and used, with the proof easily adaptable to prim maps.}

\begin{defin}
If $f_0 : M_0^n \to P^{n+k}$ and $f_1 : M_1^n \to P^{n+k}$ are two regular prim $\Sigma^{1,(1,)0}$-maps of the oriented $n$-manifolds $M_0^n$ and $M_1^n$ to the oriented manifold $P^{n+k}$, then a cobordism between them is a regular prim $\Sigma^{1,(1,)0}$ map $F: W^{n+1} \to P \times [0,1]$, where $W$ is a compact oriented $(n+1)$-manifold such that $\partial W = M_0 \underset{\partial M_0}{\cup} F^{-1}(\partial P \times [0,1]) \underset{\partial M_1}{\cup} (-M_1)$, $F|_{M_0} = f_0$ and $F|_{M_1} = f_1$.
Note that both the domain $W$ and the target $P \times [0,1]$ of $F$ may have ``corners'': $\partial M_0 \sqcup \partial M_1$ and $\partial P \times \{ 0\} \sqcup \partial P \times \{1\}$, respectively. Regularity of $F$ shall mean that near the corners $\partial M_0$ and $\partial M_1$ the map $F$ has to be the direct product of the maps $f_0|_{\partial M_0} : \partial M_0 \to \partial P \times \{ 0 \}$ and $f_1|_{\partial M_1} : \partial M_1 \to \partial P \times \{ 1 \}$ with the identity map of $[0,\varepsilon)\times [0,\varepsilon)$.
%% \par
%% The cobordism group of prim $\Sigma^{1,0}$-maps of oriented $n$-manifolds to $P$ will be denoted by  $\Prim\Sigma^{1,0}(P)$, that of prim $\Sigma^{1,1,0}$-maps of oriented $n$-manifolds to $P$ by  $\Prim\Sigma^{1,1,0}(P)$.
\end{defin}

\subsubsection*{The classifying spaces}

There exist (homotopically unique) spaces $\overline{X} \Sigma^{1,0}$ and $\overline X\Sigma^{1,1,0}$ that represent the functors
\[
\aligned
P &\longrightarrow \Prim \Sigma^{1,0}(P) \ \text{ and}\\
P &\longrightarrow \Prim \Sigma^{1,1,0}(P)
\endaligned
\]
in the sense of Brown representability theorem\footnote{In order to apply Brown's theorem directly, one has to extend these functors to arbitrary simplicial complexes (not only manifolds). This is done in \cite{Sz2}.} (see \cite{Switzer}), in particular
\[
\aligned
\Prim\Sigma^{1,0}(P) &= [P, \overline X\Sigma^{1,0}]\ \text{ and}\\
\Prim\Sigma^{1,1,0}(P) &= [P, \overline X\Sigma^{1,1,0}]
\endaligned
\]
for any compact manifold $P$ (note that we have $\Prim\Sigma^{1,0}(n+1) \cong \Prim\Sigma^{1,0}(S^{n+1})$ and $\Prim\Sigma^{1,1,0}(n+1) \cong \Prim\Sigma^{1,1,0}(S^{n+1})$).
We call the spaces $\overline X \Sigma^{1,0}$ and $\overline X \Sigma^{1,1,0}$ the \emph{classifying spaces} for prim fold and prim cusp maps respectively. This type of classifying spaces in a more general setup has been explicitly constructed and investigated earlier, see \cite{Sbornik}, \cite{LNM}, \cite{RSz}, \cite{Sz2}, \cite{Te}.

\subsubsection*{Key fibrations}

For any space $Y$ we shall denote by $\Gamma Y$ the space
\[
\Omega^\infty S^\infty Y = \lim_{q \to \infty} \Omega^q S^q Y.
\]
A crucial observation in the investigation of these classifying spaces is the existence of the so-called \emph{key fibrations} (see \cite{Sz2}), which in the present cases states that there exist Serre fibrations
$$\overline p_j: \overline X\Sigma^{1_j} \to \Gamma S^{2j+1}$$
of $\overline X \Sigma^{1_j}$ over $\Gamma S^{2j+1}$ with fibre $\overline X \Sigma^{1_{j-1}}$. In particular, we have fibrations
\begin{itemize}
\item $\overline p_1: \overline X\Sigma^{1,0} \to \Gamma S^3$ of $\overline X \Sigma^{1,0}$ over $\Gamma S^3$ with fibre $\Gamma S^1$; and
\item $\overline p_2: \overline X\Sigma^{1,1,0} \to \Gamma S^5$ of $\overline X \Sigma^{1,1,0}$ over $\Gamma S^5$ with fibre $\overline X \Sigma^{1,0}$.
\end{itemize}

\section{The spectral sequences}\label{section:SS}

\subsection{The first page}

Let us denote by $\overline X_i$ for $i = -1, 0, 1, 2$ the following spaces:
\[
\overline X_{-1} = \text{point}; \ \ \overline X_{0} = \Gamma S^1; \ \ \overline X_{1} = \overline X \Sigma^{1,0}; \ \
\overline X_{2} = \overline X \Sigma^{1,1,0}.
\]
One can define a spectral sequence with starting page
\[
\overline E_{i,j}^1 = \pi_{i + j + 1}\bigl(\overline X_i, \overline X_{i - 1}\bigr), \quad i = 0, 1, 2, \quad j=0, 1, \dots
\]
and differentials $d^r_{i,j} : \overline E_{i,j}^r \to \overline E^r_{i-r,j+r-1}$ that converges to $\pi_{n + 1}(\overline X_2) = \Prim\Sigma^{1,1,0}(n+1)$.

The existence of Serre fibrations described above implies that
\[
\overline E_{i,j}^1 \cong \pi_{i + j + 1}^{\ S} (\Gamma S^{2i + 1}) = \pi^s(j - i).
\]

\subsection{The geometric meaning of the groups and differentials of the spectral sequence}\label{section:geometric}

Let $\overline X_i$ denote the classifying space of prim $\Sigma^{1_i}$-maps, so that $\Prim\Sigma^{1_i}(n+1) = \pi_{n+1}(\overline X_i)$. The relative versions of the analogous cobordism groups of maps into the halfspace $\R^{n+1}_+ =\left\{(x_0,\dots,x_n) \in \R^{n+1} : x_0 \geq 0 \right\}$ can also be introduced and they will be isomorphic to the corresponding relative homotopy groups:

\begin{defin}
Let $(M^n, \partial M)$ be a compact $n$-manifold with (possibly empty) boundary. Let $f:(M,\partial M) \to (\R_+^{n+1},\R^n)$ be a prim $\Sigma^{1_i}$-map for which $f|_{\partial M}: \partial M \to \R^n$ is a (necessarily prim) $\Sigma^{1_j}$-map for some $j \leq i$. Such a map will be called a \emph{prim $(\Sigma^{1_i},\Sigma^{1_j})$-map} (recall that we always assume $f$ to be regular in the sense of the definition in \ref{section:singularClass}).
\par
If $f_0$ and $f_1$ are two prim $(\Sigma^{1_i},\Sigma^{1_j})$-maps of the oriented $n$-manifolds $(M_0^n,\partial M_0)$ and $(M_1^n,\partial M_1)$ to $(\R_+^{n+1},\R^n)$, then a cobordism between them is a map $F: (W^{n+1},\partial W) \to (\R_+^{n+1}\times [0,1],\R^n\times [0,1])$ (where $W$ is a compact oriented $(n+1)$-manifold) such that $\partial W = M_0 \underset{\partial M_0}{\cup} Q^n \underset{\partial M_1}{\cup} (-M_1)$, $F|_{M_0} = f_0$ and $F|_{M_1} = f_1$, for which
\begin{enumerate}[a)]
\item $Q$ is a cobordism between $\partial M_0$ and $-\partial M_1$;
\item $F|_{Q}:Q \to \R^n\times [0,1]$ is a prim $\Sigma^{1_j}$-cobordism between $f_0|_{\partial M_0} : \partial M_0 \to \R^n \times \{ 0 \}$ and $f_1|_{\partial M_1} : \partial M_1 \to \R^n \times \{ 1 \}$;
\item $F$ is a prim $\Sigma^{1_i}$-map.
\end{enumerate}
Note that both the domain $W$ and the target $\R_+^{n+1} \times [0,1]$ of $F$ have ``corners'': $\partial M_0 \sqcup \partial M_1$ and $\R^n\times \{ 0\} \sqcup \R^n \times \{1\}$ respectively. Near the corners $\partial M_0$ and $\partial M_1$ the map $F$ has to be the direct product of the maps $f_0|_{\partial M_0}$ and $f_1|_{\partial M_1}$ with the identity map of $[0,\varepsilon)\times [0,\varepsilon)$.
\par
The cobordism group of prim $(\Sigma^{1_i},\Sigma^{1_j})$-maps of oriented $n$-manifolds to $(\R_+^{n+1},\R^n)$ will be denoted by $\Prim(\Sigma^{1_i},\Sigma^{1_j})(n+1)$.
\end{defin}
Analogously to the isomorphism $\Prim\Sigma^{1_i}(n+1) \cong \pi_{n+1}(\overline X_i)$ one obtains the isomorphism
\begin{equation}\tag{\dag}\label{eq:classify}
\Prim(\Sigma^{1_i},\Sigma^{1_j})(n+1) \cong \pi_{n+1}(\overline X_i,\overline X_j).
\end{equation}

\begin{rema}
Note that $\overline X_0 = \Gamma S^1$. Indeed, $\Sigma^{1_0} = \Sigma^0$-maps are non-singular, i.e. immersions, and the classifying space for codimension $1$ oriented immersions is known to be $\Gamma S^1$ (see Fact \ref{fact:ImmFr}).
\end{rema}

The fibration $\overline p_i:\overline X_i \to \Gamma S^{2i+1}$ with fiber $\overline X_{i-1}$ induces an isomorphism of homotopy groups $(\overline p_i)_*: \pi_{n+1}(\overline X_i,\overline X_{i-1}) \to \pi_{n+1}(\Gamma S^{2i+1}) = \pi^s(n-2i)$. The geometric interpretation of this isomorphism is the following: to the cobordism class of a prim $(\Sigma^{1_i},\Sigma^{1_{i-1}})$-map $f:(M,\partial M) \to (\R_+^{n+1},\R^n)$ the mapping $(\overline p_i)_*$ associates the cobordism class of the framed immersion $f|_{\Sigma^{1_i}(f)}$ (the framing is described in Appendix $2$). Note that in particular this description implies that whenever two prim $(\Sigma^{1_i},\Sigma^{1_{i-1}})$-maps of $n$-manifolds have framed cobordant images of their $\Sigma^{1_i}$-points, then they represent the same element in $\Prim(\Sigma^{1_i},\Sigma^{1_{i-1}})(n+1)$.
\par
\paragraph*{Geometric descriptions of $d^1$.} The differential
$$
d^1: \overline E^1_{i,j} \cong \pi_{i+j+1}(\overline X_i,\overline X_{i-1}) \to \overline E^1_{i-1,j} \cong \pi_{i+j}(\overline X_{i-1},\overline X_{i-2})
$$
is simply the boundary homomorphism $\partial$ in the homotopy exact sequence of the triple $(\overline X_i,\overline X_{i-1},\overline X_{i-2})$. Composing $\partial$ with the isomorphism $(\overline p_{i-1})_*$ one can see that if $f:(M^n,\partial M) \to (\R_+^{n+1},\R^n)$ is a prim $(\Sigma^{1_i}, \Sigma^{1_{i-1}})$-map that represents the cobordism class $[f]=u\in \pi_{n+1}(\overline X_i,\overline X_{i-1}) = \overline E^1_{i,n-i} \cong \pi^s(n-2i)$, then $d^1(u) \in \pi^s((n-1)-2(i-1)) = \pi^s(n-2i+1)$ is represented by the framed immersion $f|_{\Sigma^{1_{i-1}}(f|_{\partial M})}$ in $\R^n$.
\par
There is an alternative description of $d^1$ that we shall use later as well. Let $u \in \pi_{n+1}(\overline X_i,\overline X_{i-1})$ and $f$ be a representative of $u$ as above, and let $T$ and $\wt T$ be the (immersed) tubular neighbourhoods of the top singularity strata $\Sigma^{1_i}(f)$ and $f\left(\Sigma^{1_i}(f)\right)$ in $M$ and $\R_+^{n+1}$, respectively, with the property that $f|_T$ maps $(T,\partial T)$ to $(\wt T, \partial\wt T)$. Now $f|_{\Sigma^{1_{i-1}}(f|_{\partial T})}: \Sigma^{1_{i-1}}(f|_{\partial T}) \to \partial \wt T$ is a framed immersion into $\partial \wt T$. Note that 
%% this normal framing
the normal framing of $f\left(\Sigma^{1_{i-1}}(f|_{\partial T})\right)$ inside $\partial \wt T$ defines a framing of the stable normal bundle of $f\left( \Sigma^{1_{i-1}}(f|_{\partial T})\right)$ because adding the unique outward-pointing normal vector of $\partial \wt T$ in $\R_+^{n+1}$ one obtains a normal framing in $\R_+^{n+1}$. Hence $f\left( \Sigma^{1_{i-1}}(f|_{\partial T})\right)$ with the given framing represents an element of $\pi^s(n-2i+1)$; this element is $d^1(u)$. The fact that these two descriptions of $d^1(u)$ yield the same element in $\pi^s(n-2i+1)$ follows from the fact that $f\left(\Sigma^{1_{i-1}}\left(f|_{M\setminus \mathring T}\right)\right)$ is a framed immersed cobordism between the two representatives (here we use Theorem \ref{thm:immersionCompression}).
\par
This alternative description of $d^1$ actually generalizes to the higher differentials as well (even though here we only consider $d^1$ and $d^2$; in \cite{part2} we utilize this fact for $d^r$ with greater $r$ as well).

%% \par\label{intext:diff}
\paragraph*{Geometric description of $d^2$.}
Turning to the differential $d^2$, we first give a homotopic description (an expansion of the definition, in fact). Let $u\in \pi^s(n-4) \cong \pi_{n+1}(\overline X_2,\overline X_1)=\overline E^1_{2,n-2}$ be an element such that $d^1(u)=0$. Then $u$ represents an element of the page $\overline E^2$ as well (no differential is going into the groups $\overline E^1_{2,*}$). The class $d^2(u) \in \overline E^2_{0,n-1}$ is defined utilizing the boundary homomorphism ${\partial: \pi_{n+1}(\overline X_2,\overline X_1) \to \pi_n(\overline X_1)}$ as follows: since $d^1(u) =0$, the class $\partial u \in \pi_n(\overline X_1)$ vanishes when considered in $\pi_n(\overline X_1,\overline X_0)$. Hence there is a class $y$ in $\pi_n(\overline X_0)$ whose image in $\pi_n(\overline X_1)$ is $\partial u$. The class $y$ is not unique, but the coset
$$[y] \in \pi_n(\overline X_0)/\im \left(\partial': \pi_{n+1}(\overline X_1,\overline X_0) \to \pi_n(\overline X_0)\right) =\overline E^2_{0,n-1}$$
is unique. By definition $d^2(u)=[y]$.
\par
Geometrically, if $f:(M,\partial M) \to (\R_+^{n+1},\R^n)$ represents the class $u\in \pi_{n+1}(\overline X_2,\overline X_1)$, then $d^1(u)=0$ means that $f|_{\Sigma^{1,0}(\partial f)}: \Sigma^{1,0}(\partial f) \to \R^n$ is a null-cobordant framed immersion (recall that here $\partial f$ is the restriction $f|_{\partial M}$ and is a prim fold map). This means that the classifying map $S^n \to \overline X_1$ of $\partial f$ becomes null-homotopic after being composed with $\overline p_1$, hence the classifying map itself can be deformed into the fiber $\overline X_0$. Since homotopy classes of maps into $\overline X_0$ correspond to cobordism classes of immersions, this deformation gives a {(prim $\Sigma^{1,0}$-)cobordism} between the prim $\Sigma^{1,0}$-map $\partial f: \partial M \to \R^n$ and an immersion that we will denote by $g$. The immersion $g$ is not unique, not even its framed cobordism class $[g]\in \pi^s(n-1)$ is, but its coset in $\pi^s(n-1)/\im d^1$ is well-defined. This coset is $d^2(u)$.

\begin{claim}
\label{claim:product}
For any $u \in \pi^s(n)$
\begin{enumerate}[a)]
\item $d^1_{2,n+2}(u)= d^1_{2,2}(\sigma_2)\circ u$ and $d^1_{1,n+1}(u)= d^1_{1,1}(\sigma_1)\circ u$
\item $d^2_{2,n+2}(u)$ is represented by $d^2_{2,2}(\sigma_2)\circ u$ whenever $u\in \ker d^1_{2,n+2}$.
\end{enumerate}
%%For any $n\geq 0$, $i\geq 0$ and $u \in \pi^s(n)$ we have
%%\[
%% d^1_{i,n+i}(u) = d^1_{i,i}(\ell_0) \circ u.
%%\eqno{(**)}
%%\]
\end{claim}

\noindent Note that in order for statement $b)$ to make sense, we need to show that:
\begin{itemize}
\item the ambiguity of $d^2_{2,2}(\sigma_2)\circ u$, which is $\left( \im d^1_{1,3} \right) \circ \pi^s(n)$, is contained in the ambiguity of $d^2_{2,n+2}(u)$, which is $\im d^1_{1,n+3}$. This holds by the second part of statement $a)$: we have
$$
\im d^1_{1,3} = \im d^1_{1,1} \circ \pi^s (2)
$$
and
$$
\im d^1_{1,3} \circ \pi^s(n) = \im d^1_{1,1} \circ \pi^s (2) \circ \pi^s(n) \subset \im d^1_{1,1} \circ \pi^s (n+2) = \im d^1_{1,n+3}.
$$
\item $d^2_{2,2}(\sigma_2)$ is meaningful, that is, $d^1_{2,2}(\sigma_2)=0$. This is shown later in Lemma \ref{lemma:d1cusp}.
\end{itemize}

\begin{proof} %%[Proof of $a)$.]

First we need a description of the composition product in the language of Pontryagin's framed embedded manifolds.

Given $\alpha \in \pi^s(m)$, $\beta \in \pi^s(n)$ let $(M^m, U^p)$ and $(N^n, V^{m + p})$ be representatives of $\alpha$ and $\beta$, where $M, N$ are manifolds of dimensions $m$ and $n$ immersed to $\mathbb R^{m + p}$ and $\mathbb R^{n + m + p}$, respectively, $U^p$ and $V^{m + p}$ are their framings:
$U^p = (u_1, \dots, u_p)$; $V^{m + p} = (v_1, \dots, v_{m + p})$, where $u_i, v_j$ are linearly independent normal vector fields to $M$ and~$N$.
These framings identify open tubular neighbourhoods of $M$ and $N$ with $M \times
\mathbb \R^p$ and $N \times \mathbb \R^{m + p}$.

Put the framed immersed submanifold $(M,U)$ of $\R^{m + p}$ into each fiber of the tubular neighbourhood $N \times \mathbb R^{m + p}$. We obtain $N \times M$ as a framed immersed manifold in $\R^{n + m + p}$. This is the representative of $\alpha \circ \beta$.
\par
Now we come to the proof of the first part of Claim~\ref{claim:product}$a)$; the proof of the second part is completely analogous. Let $u$ be an element in $\pi^s(n) = \pi_{n + 5}\left(\overline X_2, \overline X_1\right)$ and let $f: (M^{n+4}, \partial M) \to (\R_+^{n + 5}, \R^{n+4})$ be a prim $(\Sigma^{1_2},\Sigma^{1_1})$-map that represents a cobordism class corresponding to~$u$.
The boundary of $f$ is the prim fold map $\partial f = f\bigr|_{\partial M} : \partial M \to \R^{n+4}$. Let $\Sigma^{1,0}(\partial f)$ be the manifold of fold points of $\partial f$, and $\partial f\left(\Sigma^{1,0}(\partial f)\right)$ be its image.
Each normal fiber of $\Sigma^{1,0}(\partial f)$ in $\partial M$ is $\mathbb R^2$, and it is mapped by $\partial f$ to the corresponding normal fiber $\mathbb R^3$ of $\partial f\left(\Sigma^{1,0}(\partial f)\right)$ in $\R^{n+4}$ by the Whitney umbrella map $\sigma_1$, hence $\partial f\left(\Sigma^{1,0}(\partial f)\right)$ has a natural framing (see Appendix $2$). Then $d^1(u)$ is represented by this framed manifold $\partial f\left(\Sigma^{1,0}(\partial f)\right)$.
\par
Next we construct another representative of $d^1(u)$ using the alternative description of $d^1$. Choose small tubular neighbourhoods $T$ and $\wt T$ of $\Sigma^{1,1,0}(f)$ in $M$ and of $f\left(\Sigma^{1,1,0}(f)\right)$ in
$\R_+^{n+5}$, respectively. $\wt T$ is \emph{immersed} into $\R_+^{n+5}$, it is a $D^5$-bundle over $\Sigma^{1,1,0}(f)$. Recall that $f$ restricted to $\Sigma^{1,1,0}(f)$ is an immersion.
For simplicity of the description we suppose that it is an embedding.
We choose these tubular neighbourhoods $T$ and $\wt T$ so that $f$ maps $T$ to $\wt T$ and $\partial T$ to $\partial \wt T$.
Now the map $f\bigr|_{\partial T} : \partial T \to \partial \wt T$ is a fold map.
Its singularity is a framed manifold clearly representing $d^1(\sigma_2) \circ u$. By the two alternative descriptions of $d^1$ the framed manifold $\partial f(\Sigma^{1,0}(\partial f))$ (a representative of $d^1(u)$) represents the same framed cobordism class as the singularity of the fold map $f|_{\partial T} :\partial T \to \partial \wt T$. Hence $d^1(u)=d^1(\sigma_2) \circ u$.
\par
The proof of $b)$ is very similar. As before, for simplicity of the description of $d^2$ we suppose that  $f|_{\Sigma^{1,1,0}(f)}$ is an embedding rather than an immersion. Let $T$ and $\wt T$ as above be the tubular neighbourhoods of $\Sigma^{1,1,0}(f)$ and $f\left(\Sigma^{1,1,0}(f)\right)$ respectively. Note that (as shown in Appendix $2$) $T=\Sigma^{1,1,0}(f) \times D^4$, $\wt T = f\left(\Sigma^{1,1,0}(f)\right) \times D^5$ and the map $f|_T: (T,\partial T) \to (\wt T, \partial \wt T)$ is the product map $\left(f|_{\Sigma^{1,1,0}(f)}\right) \times \sigma_2$.
\par
Recall that in \ref{subsection:boundary} we defined a map $\sigma_2^*: D^4 \underset{L}{\cup} W^4 \to D_2^5$ such that $\partial \sigma_2^*$ is an immersion $V^3 \to S^4 = \partial D^5_2$ that represents $d^2(\sigma_2)$. Now we define a map
$$
f^* = f|_{\Sigma^{1,1,0}(f)} \times \sigma_2^*: \Sigma^{1,1,0}(f) \times \left(D^4 \underset{L}{\cup} W^4 \right) \to f\left(\Sigma^{1,1,0}(f)\right) \times D_2^5. %% \overset{\mathrm{def}}{=} \wt T_2.
$$
We will denote by $T_2$ the source manifold of $f^*$ and by $\wt T_2$ the target manifold, which is an enlarged tubular neighbourhood of $f\left( \Sigma^{1,1,0}(f)\right)$ in $\R_+^{n+5}$. Then $\partial f^* : \partial T_2 = \Sigma^{1,1,0}(f) \times V^3 \to \partial \wt T_2 = f\left( \Sigma^{1,1,0}(f)\right) \times S^4 \times \left\{ 1 \right\}$ is an immersion that represents $d^2(\sigma_2) \circ u$.
\par
We claim that this immersion can be extended to a proper, regular cusp map $\widehat{f^*}$ (of some compact $(n+4)$-manifold with boundary) into the entire $\R_+^{n+5}$ without changing the singular set. Indeed, the source manifold of $\partial f^*$ has dimension $n+3$, hence the image of $\partial f^*$ is an $(n+3)$-dimensional compact complex (denote it by $K$), and it can be covered by a small neighbourhood $U$ in $\partial \wt T_2$ of $K$ whose closure $\overline U$ is a compact manifold with non-empty boundary. By Theorem \ref{thm:multicompression}, there exists a deformation of $\overline U$ (equipped with the outward-pointing normal vector field) within $\R_+^{n+5}$ with time derivative nowhere tangent to the image of $\overline U$ that takes $\overline U$ into $\R^{n+4}=\partial \R^{n+5}_+$ (and the normal vector field into the outward-pointing vector field of $\R^{n+4}$). The trace of this deformation glued along $\partial f^*$ to $f^*$ gives an extension $\widehat{f^*}$ whose set of cusp points is the same as that of $f^*$ and in particular represents $u$ in $\mathcal G$.
\par
This construction combined with Theorem \ref{thm:immersionCompression} shows that $\partial f^*$ and $\partial \widehat{f^*}$ are cobordant as framed immersions and therefore represent the same element in $\mathcal G$; the statement $b)$ follows since $\partial f^*$ represents $d^2(\sigma_2) \circ u$ and $\partial \widehat{f^*}$ represents $d^2(u)$.
\end{proof}

\subsection{Calculation of the first page of the spectral sequence for prim maps}

\smallskip
\noindent
\parbox{61.5mm}{\vspace*{-36mm}
\begin{tabular}{@{}r||c|c|c|c}
&&& & \\
\cline{1-5}
$3$ & ${\Z}_{24}\rule{4pt}{0pt} $ & \hspace*{-7.5mm}$\overset{\textstyle d_{1,3}^1}{\longleftarrow} {\Z}_2$ & ${\Z}_2$ & \rule{5mm}{0pt} \\
\cline{1-5}
$2$ & ${\Z}_{2}\rule{6pt}{0pt} $ & \hspace*{-5.75mm}$\underset{\cong}{\overset{\textstyle d_{1,2}^1}{\longleftarrow}} {\Z}_2\rule{6pt}{0pt}$
& \hspace*{-6mm}$\underset{0}{\overset{\textstyle d_{2,2}^1}{\longleftarrow}} {\Z}$ & \\
\cline{1-5}
$1$ & ${\Z}_{2}\rule{6pt}{0pt} $ & \hspace*{-8.8mm}$\overset{\textstyle d_{1,1}^1}{\longleftarrow} {\Z}$ &  & \\
\cline{1-5}
$j = 0$ & ${\Z}$ &  &  & \\
\hhline{=||=|=|=|=}
& $i = 0$ & 1 & 2 &
\end{tabular}}%
\parbox[b]{65mm}{Recall that
\[
\aligned
\overline E_{i,j}^1 &= \pi_{i+ j+1}(\overline X_i, \overline X_{i - 1})\\
& = \pi_{i+j+1} (\Gamma S^{2i + 1}) = \pi^s(j-i).
\endaligned
\]
Hence on the diagonal $j = i$ we have $\pi^s(0) = {\Z}$ with generator $\iota_i$ in $\overline E^1_{i,i}$ represented by the map $\sigma_i: (D^{2i},S^{2i-1}) \to (D^{2i+1},S^{2i})$ that has an isolated $\Sigma^{1_i}$ singularity at the origin.
On the line $j = i + t$ we have $\pi^s(t)$.}

\bigskip
The value of $d_{1,1}^1(\iota_1)$ is nothing else but $[\partial \sigma_1] = \eta \in \pi^s(1) = \Z_2$.

By Claim \ref{claim:product} we have $d_{1,2}^1(\eta) = d^1_{1,1}(\iota_1) \circ \eta = \eta \circ \eta \neq 0$ in $\pi^s(2)$ (here and later we refer the reader to \cite[Chapter XIV]{To} for the information that we need about the composition product). Hence $d^1_{1,2}$ is an isomorphism and it follows that $d_{2,2}^1$ is zero (since $d^1_{1,2}\circ d^1_{2,2} =0$). In particular, we have obtained the following lemma:

\begin{lemma}\label{lemma:d1cusp}
The class $d^1(\iota_2)$, represented by the image $\sigma_2(\gamma)$ of the fold singularity curve on the boundary of the isolated cusp $\sigma_2: \R^4 \to \R^5$, vanishes.
\end{lemma}

In Appendix $1$ we give an independent, elementary proof for this statement.

\subsection{The second page $(\overline E_{i,j}^2, \overline d_{i,j}^2)$ for prim maps}

The differential $d_{1,3}^1 : {\Z}_2 \to {\Z}_{24}$ maps the generator $\eta \circ \eta$ of $\pi^s(2)$ to $d^1_{1,1}(\iota_1) \circ \eta\circ\eta =\eta \circ \eta \circ \eta$ and that is not zero (\cite[Theorem 14.1]{To}).
 Hence the group $\overline E_{0,3}^2$ is ${\Z}_{24} / {\Z}_2 = {\Z}_{12}$.
\par
\vspace*{0mm}
\vbox{
$$
\xymatrix{
& & & & \\
3 & \Z_{12} & 0 & ? & \\
2 & 0 & 0 & \Z \ar[ull]_{d^2_{2,2}} & \\
1 & 0 & \Z & & \\
j=0 & \Z & & \\
& i=0 & 1 & 2 & 
}
$$
\vspace*{-72mm}
$$
\hspace*{18pt}\begin{array}{c||c|c|c|c}
\hspace*{20pt} & \hspace*{42pt} & \hspace*{30pt}& \hspace*{28pt} & \hspace*{12pt} \\[4pt]
\cline{1-5}
 & & & & \\[24pt]
\cline{1-5}
 & & & & \\[24pt]
\cline{1-5}
 & & & & \\[24pt]
\cline{1-5}
 & & & & \\[24pt]
\hhline{=||=|=|=|=}
  & & & & \\[24pt]
\end{array}
$$
}
\par
%% \footnotetext{The group denoted by the question mark in the spectral sequence is in fact $0$, see \cite{part2}.}
Now we compute the differential $d_{2,2}^2 : {\Z} \to {\Z}_{12}$.
Note that this is precisely the computation of the cobordism class of the framed immersion $\partial'\sigma_2$ of the $3$-manifold $V^3$ considered in Example $2$ in \ref{subsection:boundary}.

\begin{lemma}
\label{lem:1}
$d_{2,2}^2: {\Z} \to {\Z}_{12}$ maps the generator $\iota_2$ of $\overline E^2_{2,2} \cong {\Z}$ into an element of order~$6$.
%% (So $E_{0,3}^3 = E_{0,3}^2 / \text{\rm im }d_{2,2}^2$ will be ${\Z}_{12}/{\Z}_2 = {\Z}_6$.)
\end{lemma}

\begin{proof}
$\overline E^1_{2,2} \cong \pi_5(\overline X_2,\overline X_1)=\pi_5(\Gamma S^5)=\pi^s_5(S^5) = \Z$. Since $d^1_{2,2}: E^1_{2.2} \to E^1_{1,2}$ is identically zero, $\overline E^2_{2,2}=\overline E^1_{2,2}$.
\par
Consider the following commutative diagram with exact row and column:
$$
\hskip-1cm
\xymatrix{
& & \pi_4(\overline X_0)/\im \left( \pi_5(\overline X_1,\overline X_0) \overset{\partial}{\to} \pi_4(\overline X_0) \right) \ar@{>->}[d]\\
\quad\quad\quad\quad\pi_5(\overline X_2) \ar[r]^{\varphi} & \pi_5(\overline X_2,\overline X_1) \cong \Z \ar[dr]_{d^1} \ar[r] \ar[ur]^{d^2} & \pi_4(\overline X_1) \ar[d] \\
& & \pi_4(\overline X_1,\overline X_0)
}
$$
The generator $\iota_2$ of $\pi_5(\overline X_2,\overline X_1) \cong \Z$ is represented by the cusp map $\sigma_2:(D^4,L^3) \to (D^5,S^4)$ (using the notation of Example $2$). Simple diagram chasing shows that the order of $d^2(\iota_2)$ is equal to the order of $\Coker \varphi$. The latter is the minimal positive algebraic number of cusp points of prim cusp maps of oriented closed $4$-manifolds into $\R^5$. Indeed, $\varphi$ assigns to the class of a map $f:M^4 \to \R^5$ the algebraic number of its cusp points. This minimal number of cusps is known to be $6$, see \cite[Theorem 4]{Sz1}.

\end{proof}
%% \hfill $\square$

Applying Claim \ref{claim:product} $b)$ we immediately get:

\begin{cor}
On the $3$-torsion part, the differential $d^2_{2,n+2}$ acts as the homomorphism $\alpha_1$ (as defined in \ref{section:notation}) up to sign.
\end{cor}

\subsection{Computation of the cobordism group of prim fold maps of oriented $n$-manifolds to $\mathbb R^{n + 1}$}

\begin{theorem}\label{thm:primCob}
\begin{enumerate}[a)]
\item $\Prim\Sigma^{1,0}(n+1) \underset{\mathcal C_2}{\cong} \pi^s(n) \oplus \pi^s(n-2)$.
\item $\Prim\Sigma^{1,1,0}(n+1) \underset{\mathcal C_{\{2,3\}}}{\cong} \pi^s(n) \oplus \pi^s(n-2) \oplus \pi^s(n-4)$.
\end{enumerate}
\end{theorem}

\begin{proof}
We have seen that the spectral sequences computing $\Prim\Sigma^{1,0}(n+1)$ and $\Prim\Sigma^{1,1,0}(n+1)$ degenerate modulo $\mathcal C_2$ and modulo $\mathcal C_{\{2,3\}}$ respectively, because $d^1_{1,n+1}$ is multiplication by the order $2$ element $\eta$ and $d^2_{2,n+2}$ is multiplication by an element of order $6$.
\par
The fact that the cobordism groups $\Prim\Sigma^{1,0}(n+1)$ and \linebreak $\Prim\Sigma^{1,1,0}(n+1)$ are direct sums (modulo $2$- and $3$-primary torsion parts) can be shown in the same way as in \cite[Theorem B]{Sz3}. Namely, the homotopy exact sequence of the fibration $\overline p_1: (\overline X_1,\overline X_0) \to \Gamma S^3$
$$
\pi_{n+1}(\overline X_0) \to \pi_{n+1}(\overline X_1) \overset{(\overline p_1)_*}{\to} \pi_{n+1}(\Gamma S^3)
$$
has a $2$-splitting $s$, that is, there is a homomorphism $s:\pi_{n+1}(\Gamma S^3) \to \pi_{n+1}(\overline X_1)$ such that $(\overline p_1)_* \circ s$ is the multiplication by $2$. The construction of $s$ goes as follows: choose an immersion $S^2 \looparrowright \R^4$ with normal Euler number $2$. Then its generic projection to $\R^3$ will be a map $\psi: S^2 \to \R^3$ with finitely many Whitney umbrella points that inherit a sign from the orientation of the kernel bundle, and the algebraic number of these points will be $2$ (see \cite[Proposition 2.5.]{SaekiSakuma}). Now choosing any framed immersion $q: Q^{n-2} \looparrowright \R^{n+1}$ that represents an element $[q]$ in $\pi_{n+1}(\Gamma S^3) \cong \pi^s(n-2)$, the framing of $Q$ defines a prim fold map $Q \times S^2 \overset{id \times \psi}{\to} Q \times \R^3 \looparrowright \R^{n+1}$. Its class will be $s([q])$. The existence of the $2$-splitting map $s$ implies part $a)$.
\par
The existence of an analogous $6$-splitting of the homotopy exact sequence of the fibration $\overline p_2: (\overline X_2,\overline X_1) \to \Gamma S^5$ is shown in \cite[Lemma 4]{Sz3}.
%% (see also \cite[Theorem 1.1]{Lippner}).
It shows that $\Prim\Sigma^{1,1,0}(n+1) \underset{\mathcal C_{\{2,3\}}}{\cong} \Prim\Sigma^{1,0}(n+1) \oplus \pi^s(n-4)$ and together with part $a)$ proves part $b)$.
\end{proof}

\begin{theorem}
Let $\eta_n: \pi^s(n) \to \pi^s(n+1)$ be the homomorphism $x \mapsto \eta \circ x$. Then the following sequence is exact:
$$
0 \to \Coker \eta_{n-1} \to \Prim\Sigma^{1,0}(n+1) \to \ker\eta_{n-2} \to 0
$$
\end{theorem}

\begin{proof}
In the homotopy exact sequence of the pair $(\overline X_1,\overline X_0)$ the boundary homomorphism is the differential $d^1_{1,n+1}$, which by Claim \ref{claim:product} $a)$ is just the corresponding homomorphism $\eta_n$. The statement follows immediately.
\end{proof}

Recall that for any abelian group $G$ we denote by $(G)_3$ its $3$-primary part. While $\Prim \Sigma^{1,1,0}(n+1)$ is computed by Theorem \ref{thm:primCob} modulo its $2$- and $3$-primary parts, we can also compute the $3$-primary part (up to a group extension).

\begin{theorem}\label{thm:cuspCob}
The $3$-primary part of $\Prim\Sigma^{1,1,0}(n+1)$ fits into the short exact sequence
\begin{align*}
0 & \to \left( \Coker \left( \alpha_1: \pi^s(n-3)  \to \pi^s(n) \right) \right)_3 \oplus (\pi^s(n-2))_3 \to \\
& \quad \to \left( \Prim\Sigma^{1,1,0}(n+1) \right)_3 \to \left( \ker (\alpha_1: \pi^s(n-4) \to \pi^s(n-1)) \right)_3 \to 0\\
\end{align*}
\end{theorem}

\begin{proof}
The spectral sequence $\overline E^r_{i,j}$ converges to $\Prim\Sigma^{1,1,0}(n+1)$ and stabilizes at page $3$. Recall that on the $3$-primary part $d^2$ can be identified up to sign with the homomorphism $\alpha_1$ (Corollary of Lemma \ref{lem:1}). Hence the $3$-primary parts $\left(\overline E^3_{i,j}\right)_3$ of the groups $\overline E^3_{i,j} \cong \overline E^\infty_{i,j}$ are the following:
\begin{equation}\tag{*}\label{eq:primE}
\begin{aligned}
\left( \overline E^3_{0,j} \right)_3 &= \left(\Coker (\alpha_1: \pi^s(j-3) \to \pi^s(j))\right)_3\\
\left( \overline E^3_{1,j} \right)_3 &= (\pi^s(j-1))_3\\
\left( \overline E^3_{2,j} \right)_3 &= \left(\ker (\alpha_1: \pi^s(j-2) \to \pi^s(j+1))\right)_3
\end{aligned}
\end{equation}
By general properties of spectral sequences it holds that if we define the groups
\begin{align*}
F_{2,n} & = \Prim\Sigma^{1,1,0}(n+1) = \pi_{n+1}(\overline X_2)\\
F_{1,n} & = \im \left( \Prim\Sigma^{1,0}(n+1) \to \Prim\Sigma^{1,1,0}(n+1) \right) =\\
& \qquad = \im \left( \pi_{n+1}(\overline X_1) \to \pi_{n+1}(\overline X_2) \right)\\
F_{0,n} & = \im \left( \pi^s(n) \to \Prim\Sigma^{1,1,0}(n+1) \right) = \\
& \qquad = \im \left( \pi_{n+1}(\overline X_0) \to \pi_{n+1}(\overline X_2) \right)\\
\end{align*}
then
\begin{align*}
F_{2,n}/F_{1,n} & = \overline E^\infty_{2,n-2} \\
F_{1,n}/F_{0,n} & = \overline E^\infty_{1,n-1} \\
F_{0,n} & = \overline E^\infty_{0,n}
\end{align*}
We will show that the exact sequence
\begin{equation}\tag{$\ddagger$}\label{eq:SES}
0 \to (F_{0,n})_3 \to (F_{1,n})_3 \to (F_{1,n}/F_{0,n})_3 \to 0
\end{equation}
splits and hence $(F_{1,n})_3 \cong (F_{0,n})_3 \oplus (F_{1,n}/F_{0,n})_3$. Then the exact sequence
$$
0 \to (F_{1,n})_3 \to (F_{2,n})_3 \to (F_{2,n}/F_{1,n})_3 \to 0
$$
can be written as
$$
0 \to (F_{0,n})_3 \oplus (F_{1,n}/F_{0,n})_3 \to (F_{2,n})_3 \to (F_{2,n}/F_{1,n})_3 \to 0,
$$
and substituting \eqref{eq:primE} gives us the statement of Theorem \ref{thm:cuspCob}.
\par
It remains to show that \eqref{eq:SES} splits. Consider the following commutative diagram:
$$
\xymatrix{
&&&\pi_{n+2}(\overline X_2,\overline X_1) \ar[d]^{\partial = d^1_{2,n-1}} \\
\pi_{n+1}(\overline X_0) \ar[r] \ar@{->>}[d] & \pi_{n+1}(\overline X_1) \ar[rr] \ar@{->>}[d]^{pr}&& \pi_{n+1}(\overline X_1, \overline X_0) \ar[dd]^{i} \ar@/_1pc/[ll]_s\\
F_{0,n} \rule[-8pt]{0pt}{0pt} \ar@{>->}[d] \ar[r] & F_{1,n} \rule[-8pt]{0pt}{0pt} \ar@{>->}[d] \ar[r] & F_{1,n}/F_{0,n} \ar@{-->}[ul]_{\hat s} \ar[dr]^{j}
\ar@{-->}[ur]^{\hat j}& \\
\pi_{n+1}(\overline X_2) \ar@{=}[r]& \pi_{n+1}(\overline X_2) \ar[rr]^{r} & & \pi_{n+1}(\overline X_2,\overline X_0)
}
$$
Consider the composition map  $F_{1,n} \rightarrowtail \pi_{n+1}(\overline X_2) \overset{r}{\to} \pi_{n+1}(\overline X_2, \overline X_0)$. Its kernel is the intersection $\ker r \cap F_{1,n}$; but $\ker r$ is the image of $\pi_{n+1}(\overline X_0)$, which is $F_{0,n}$. Hence the map $r$ defines uniquely a map $j : F_{1,n}/F_{0,n} \to \pi_{n+1}(\overline X_2,\overline X_0)$. Its image $\im j$ is a subset of $\im i$ due to the commutativity of the right-hand square.
\par
By Claim \ref{claim:product}, in the (exact) rightmost column the map $\partial$ (which can be identified with $d^1_{2,n-1}$) acts on $\pi_{n+2}(\overline X_2,\overline X_1) \cong \pi_{n+2}(\Gamma S^5)$ by composition from the left by $\partial
[\sigma_2]$, which is zero. Hence the map $i$ is injective. Consequently, the map $j$ can be lifted to
a map $\hat j: F_{1,n}/F_{0,n} \to \pi_{n+1}(\overline X_1,\overline X_0)$ and composing it with the $2$-splitting $s$ gives us a map $\hat s
=s \circ \hat j : F_{1,n}/F_{0,n} \to \pi_{n+1}(\overline X_1)$ such that $pr \circ \hat s$ is a $2$-splitting of the short exact sequence $0 \to F_{0,n} \to F_{1,n} \to F_{1,n}/F_{0,n} \to 0$. This proves that on the level of $3$-primary parts this extension is trivial, as claimed.

%In the (exact) rightmost column, analogously to Claim \ref{claim:product}, the map $\partial$ sends the element $\kappa \in \pi_{n+2}(\overline X_2,\overline X_1) \cong \pi_{n+2}(\Gamma S^5)$ to $\partial [\sigma_2] \circ \kappa$. But $\partial [\sigma_2]$ represents $d^1\iota_2 =0$, so $\partial$ vanishes identically. Hence the map $i$ is injective, and one can (uniquely) define the map $\hat s: F_{1,n}/F_{0,n} \to \pi_{n+1}(\overline X_1)$ such that $s=\hat s \circ i$. Following $\hat s$ into $F_{1,n}$, we get a $2$-splitting, and all the involved groups are $3$-primary, so we have a direct sum as claimed.

\end{proof}

\subsection{The spectral sequence for arbitrary (not necessarily prim) cusp maps}

There are classifying spaces for the cobordisms of codimension $1$ cooriented arbitrary (not necessarily prim) $\Sigma^{1_i}$-maps as well. We denote by $X_i=X\Sigma^{1_i}$ the classifying space of such $\Sigma^{1_i}$-maps with the convention that $X_{-1} = *$. Here we will mostly be interested in $X_0$, $X_1$ and $X_2$. The filtration $X_{-1} \subset X_0 \subset X_1 \subset X_2$ gives again a spectral sequence with $E^1_{i,j}= \pi_{i+j+1}(X_i,X_{i-1})$ for $i=0,1,2$, $j=0,1,\dots$. Analogously to the fibrations $\overline p_i$ we have fibrations
\begin{itemize}
\item $p_1: X_1 \to \Gamma T(2\varepsilon^1 \oplus \gamma^1) = \Gamma S^2 \R P^\infty \text{ with fibre }X_0$
\item $p_2: X_2 \to \Gamma T(3\varepsilon^1 \oplus 2\gamma^1) = \Gamma S^3(\R P^\infty/\R P^1)$ with fibre $X_1$ (for the identification of $T(2\gamma^1)$ and $\R P^\infty / \R P^1$ see e.g. \cite[Example 1.7, ch. 15]{Husemoller}).
\end{itemize}
Here $\gamma^1$ and $\varepsilon^1$ are the canonical and the trivial line bundles over $\R P^\infty$, respectively, and $T$ stands for Thom space (recall that $\Gamma=\Omega^\infty S^\infty$). Note that $X_0=\overline X_0 = \Gamma S^1$.
\par
Observe that the base spaces of $p_i$ are different from those of $\overline p_i$. This change is due to the fact that while the normal bundles of the singularity strata for a prim map are trivial and even canonically trivialized (see Appendix $2$), for arbitrary cooriented codimension $1$ Morin maps they are direct sums of not necessarily trivial line bundles (see \cite[Theorem 6]{RSz} and the definition of $G^{SO}$ that precedes it). The bundles $2\varepsilon^1 \oplus \gamma^1$ and $3\varepsilon^1 \oplus 2 \gamma^1$ are the universal normal bundles in the target of the fold and cusp strata respectively.

\begin{prop}
\begin{enumerate}[a)]
\item $E^1_{1,j}\in \mathcal C_2$.
\item $E^1_{i,j} \cong \overline E^1_{i,j}$ modulo $\mathcal C_2$ for $i=0,2$.
\end{enumerate}
\end{prop}
\par
\begin{proof}
\begin{enumerate}[a)]
\item Since $H_*(\R P^\infty; \Z_p) =0$ for $p$ odd, the Serre-Hurewicz theorem implies that $\pi^s_*(\R P^\infty) \in \mathcal C_2$ and therefore
$$
E^1_{1,j} \cong \pi_{j+2}(\Gamma S^2 \R P^\infty) = \pi^s_j(\R P^\infty) \in \mathcal C_2.
$$
\item Since the inclusion $S^2 = \R P^2/\R P^1 \hookrightarrow \R P^\infty/\R P^1$ induces isomorphism of $\Z_p$-homologies (the groups $H_*(\R P^\infty/\R P^2;\Z_p)$ all vanish) for $p$ odd, we have
$$
E^1_{2,j} \cong \pi_{j+3}(\Gamma S^3(\R P^\infty/\R P^1)) \underset{\mathcal C_2}{\cong} \pi_{j+3}(\Gamma S^5) \cong \overline E^1_{2,j}.
$$
We also have $X_0=\overline X_0$ and consequently
$$
E^1_{0,j} \cong \overline E^1_{0,j}.
$$
\end{enumerate}
\end{proof}

\subsection{Computation of the cobordism group of (arbitrary) cusp maps}

\begin{proof}[Proof of Theorem \ref{th:1}]
The natural forgetting map $\overline E^1_{i,j} \to E^1_{i,j}$ induces a $\mathcal C_2$-isomorphism for $i=0,2$, and $E^1_{1,j} \in \mathcal C_2$. Since the $d^1$ differential is trivial modulo $\mathcal C_2$ for both spectral sequences, the map $\overline E^2_{i,j} \to E^2_{i,j}$ is a $\mathcal C_2$-isomorphism for $i=0,2$.
\par
Hence the differential $d^2$ restricted to the $3$-primary part can be identified in the two spectral sequences, and we obtain that $\left(E^\infty_{i,j}\right)_3 = \left(E^3_{i,j}\right)_3 \cong \left(\overline E^3_{i,j}\right)_3$ for $i=0,2$ (but not for $i=1$). The statement of the theorem follows analogously to Theorem \ref{thm:cuspCob}.
\end{proof}

\section*{Appendix 1: an elementary proof of Lemma \ref{lemma:d1cusp}}

In this Appendix we give an elementary and independent proof of the fact that the curve of folds on the boundary sphere of an isolated cusp map $\sigma_2:
\R^4 \to \R^5$ with the natural framing represents the trivial element in $\pi^s(1) = {\Z}_2$.

$\sigma_2 : \R^4 \rightarrow \R^5$, $\sigma_2(t_1,t_2,t_3,x) = (t_1,t_2,t_3,t_1x+t_2x^2,t_3x+x^3)$

\[
\mathrm{d}\sigma_2 = 
\begin{bmatrix} 
1 & 0 & 0 & 0 \\
0 & 1 & 0 & 0 \\
0 & 0 & 1 & 0 \\
x & x^2 & 0 & t_1+2xt_2 \\
0 & 0 & x & t_3+3x^2
\end{bmatrix}
\]

The set of singular points of $\sigma_2$ is $\Sigma = \{ (-2xt_2, t_2, -3x^2, x) \mid t_2, x \in \R \}$, its image is $\tilde{\Sigma} = \sigma_2(\Sigma) = \{ (-2xt_2, t_2, -3x^2, -t_2x^2, -2x^3) \mid t_2, x \in \R \}$.

For a point $p \in \R^4 \setminus \Sigma$ the vector $n(p) = (-x(t_3+3x^2), -x^2(t_3+3x^2), x(t_1+2xt_2), t_3+3x^2, -(t_1+2xt_2))$ is non-zero and orthogonal to the columns of $\mathrm{d}\sigma_2$, so it is a normal vector of the immersed hypersurface $\sigma_2(\R^4 \setminus \Sigma) \subset \R^5$ at $\sigma_2(p)$.

$\tilde{\Sigma} \setminus \{ 0 \}$ is an embedded surface in $\R^5$, and it has a canonical framing:

Through each point $p = (-2xt_2, t_2, -3x^2, x) \in \Sigma \setminus \{ 0 \}$ we can define a curve
$$
p_{\varepsilon} = (-2xt_2, t_2, -3x^2-\varepsilon^2, x+\varepsilon)
$$
such that $p_0 = p$, $\frac{\partial p_{\varepsilon}}{\partial \varepsilon} (0) = (0,0,0,1) \in \ker \mathrm{d}\sigma_2$, and $\sigma_2(p_{\varepsilon}) = \sigma_2(p_{-\varepsilon}) = q_{\varepsilon^2}$, where
$$
q_{\delta} = (-2xt_2, t_2, -3x^2-\delta, -t_2(x^2-\delta), -2x^3 + 2 \delta x).$$
(Note that by taking this curve for each $p$ we have defined an orientation of the kernel line bundle of $\mathrm{d}\sigma_2$.)

The first vector of the framing is the tangent vector of the image curve $q_{\delta}$: 
\[
v_1 = \frac{\partial q_{\delta}}{\partial \delta} (0) = (0, 0, -1, t_2, 2x).
\]

Since $\sigma_2(p_{\varepsilon}) = \sigma_2(p_{-\varepsilon}) = q_{\varepsilon^2}$, in this point we have defined two normal vectors of $\sigma_2(\R^4 \setminus \Sigma)$, namely
\begin{align*}
n(p_{\pm \varepsilon}) &= \pm \varepsilon(-6x^2 - 2\varepsilon^2, -6x^3 - 10x\varepsilon^2, 2xt_2, 6x, -2t_2) +\\
& \quad +\varepsilon^2(-8x, -14x^2 - 2\varepsilon^2,2t_2,2,0).
\end{align*}
The sum and the difference of these vectors are
\begin{align*}
n(p_{\varepsilon}) + n(p_{-\varepsilon}) &= 2\varepsilon^2(-8x, -14x^2 - 2\varepsilon^2,2t_2,2,0) \text{ and}\\
n(p_{\varepsilon}) - n(p_{-\varepsilon}) &= 2\varepsilon(-6x^2 - 2\varepsilon^2, -6x^3 - 10x\varepsilon^2, 2xt_2, 6x, -2t_2).
\end{align*}
The last two vectors of the framing are the limits of (the directions of) these vectors:
\[
\begin{aligned}
v_2 &= (-3x^2, -3x^3, xt_2, 3x, -t_2) \\
v_3 &= (-4x, -7x^2, t_2, 1, 0)
\end{aligned}
\]

The following Claim implies that the framed curve $\sigma_2(\gamma)$ is null-cobordant (recall that $\gamma = \sigma_2^{-1}(S^4) \cap \Sigma$).

\begin{claim}
There is a smooth embedding $F : D^2 = \left\{ (t_2,x) \mid t_2^2 + x^2 \leq 1\right \} \rightarrow \R^5$ and a framing of $F(D^2)$ that extends $\sigma_2 \circ i$ and the canonical framing of $\tilde{\Sigma} \setminus \{ 0 \}$ restricted to $\sigma_2 \circ i(S^1)$, where $i : S^1 = \left\{ (t_2,x) \mid t_2^2 + x^2 = 1 \right\} \rightarrow \Sigma$, $i(t_2,x) = (-2xt_2,t_2,-3x^2,x)$.
\end{claim}

\begin{proof}
We define such an $F$ and a framing:
\[
\begin{aligned}
F(t_2,x) &= (-2xt_2,t_2,-3x^2,-t_2x^2,2x(t_2^2-1)) \\
v_1 &= (0,0,-1,t_2,2x) \\
v_2 &= (3-3t_2^2-6x^2,-3x^3,xt_2,3x,-t_2) \\
v_3 &= (-4x,-7x^2,t_2,1,0) 
\end{aligned}
\]

These are smooth, and in the case $t_2^2+x^2=1$ they coincide with the previously defined map and framing. It is easy to check that $F$ is injective. We need to prove that the differential of $F$ is injective, and the vectors really form a framing, i.e.\ that the partial derivatives of $F$ and $v_1, v_2, v_3$ are linearly independent. Equivalently, the following matrix should be non-singular:
\[
M = 
\begin{bmatrix} 
-2x & 1 & 0 & -x^2 & 4xt_2 \\
t_2 & 0 & 3x & xt_2 & 1-t_2^2 \\
0 & 0 & -1 & t_2 & 2x \\
3-3t_2^2-6x^2 & -3x^3 & xt_2 & 3x & -t_2 \\
-4x & -7x^2 & t_2 & 1 & 0
\end{bmatrix}
\]

\begin{align*}
\det M &= 180x^8 + 568x^6t_2^2 + 323x^4t_2^4 + 120x^6 - 197x^4t_2^2 + 8x^2t_2^4 + 3t_2^6 +\\
&\qquad\qquad +51x^4 - 12x^2t_2^2 - 2t_2^4 + 24x^2 - 2t_2^2 + 3\\
&= 50(2x^2t_2^2-x^2)^2 + 6(xt_2^2-x)^2 + 2(t_2^3-t_2)^2 + 2(t_2^2-1)^2 + \\
&\qquad\qquad +180x^8 + 568x^6t_2^2 + 123x^4t_2^4 + 120x^6 + 3x^4t_2^2 + 2x^2t_2^4 +\\
&\qquad\qquad +t_2^6 + x^4 + 18x^2 + 1 \\
&> 0.
\end{align*}
Therefore $M$ is always non-singular, and the proof is complete.
\end{proof}

\section*{Appendix 2: the natural framing on the image of the manifold formed by the $\Sigma^{1_r}$-points of a cooriented prim map}

Let us consider the map
\begin{align*}
\sigma_r : (\R^{2r},0) &\to (\R^{2r+1},0),\\
(t_1,\dots,t_{2r-1},x) & \mapsto (t_1,\dots,t_{2r-1},z_1,z_2)\\
z_1 &= t_1 x+ \dots + t_r x^r\\
z_2 &= t_{r+1} x + \dots + t_{2r-1} x^{r-1} + x^{r+1}
\end{align*}
the Morin normal form of a map with an isolated $\Sigma^{1_r}$-point at $0$. Denote by $\Delta_{r+1}$ the set of $(r+1)$-tuple points, i.e. the points $\{ q\in \R^{2r+1} : \sigma_r^{-1}(q) \text{ consists of }r+1\text{ different points} \}$, and let $\overline \Delta$ be the closure of $\Delta_{r+1}$.

\begin{lemma}\label{lemma:polyroots}
\begin{enumerate}[a)]
\item The set $\overline \Delta$ is contained in the linear subspace $\mathcal P$ of $\R^{2r+1}$ of dimension $r$ defined by the equations $z_1=t_1=\dots=t_r=0$.
\item Identifying the $r$-tuple $(t_{r+1},\dots,t_{2r-1},z_2)$ with the polynomial $-z_2+t_{r+1}x + \dots + t_{2r-1} x^{r-1} + x^{r+1}$, the points of $\overline \Delta$ correspond precisely to the polynomials whose roots are all real.
\end{enumerate}
\end{lemma}

\begin{proof}
%% \begin{enumerate}[a)]

%% \item %% a
$a)$
Assume that $\sigma_r$ maps the $r+1$ different points
$$
p_j = (t_1^{(j)}, \dots, t_{2r-1}^{(j)},x_j),\quad j=1,\dots,r+1,
$$
to the same point
$$
q=(t_1,\dots,t_{2r-1},z_1,z_2).
$$
Then necessarily $t_i^{(j)}=t_i$ for all $i=1,\dots,2r-1$ and $j=1,\dots,r+1$. Since the points $p_j$ are pairwise different, the values $x_1,\dots,x_{r+1}$ are also pairwise different. But they are roots of the polynomial
$$
-z_1+t_1x+ \dots + t_rx^r,
$$
which has degree at most $r$, hence this polynomial must identically vanish. Consequently we have $z_1=t_1=\dots=t_r=0$ on $\Delta_{r+1}$ and therefore also on $\overline \Delta$ as claimed.

\par\medskip
$b)$
As in part $a)$, in the preimage of $\Delta_{r+1}$ the coordinates $x_1,\dots,x_{r+1}$ are the roots of the polynomial
$$
-z_2+t_{r+1}x+\dots+t_{2r-1}x^{r-1}+x^{r+1}.
$$
This immediately implies part $b)$ since for any polynomial that corresponds to a point in $\overline \Delta$ any sum-preserving perturbation of its (real) roots that makes them distinct gives a polynomial that corresponds to a point in $\Delta_{r+1}$.
\end{proof}

In the next lemma we describe $\overline \Delta$ by identifying it with the orthant
$$
\R^{r}_\angle = \{ (u_1, \dots, u_{r}): u_j \geq 0 \text{ for all }j=1,\dots,r\}.
$$
Moreover, we relate the natural stratification on $\mathcal P$ -- where strata are sets of polynomials with the same multiplicities of the ordered roots -- to that on $\R^r_\angle$, where the strata are the faces.

\begin{lemma}
$\overline \Delta$ is homeomorphic to the orthant $\R^{r}_\angle$, and a homeomorphism $\varphi : \R^r_\angle \to \overline \Delta$ can be chosen to map the natural stratification of $\R^r_\angle$ bijectively onto that of $\overline \Delta$ inherited from $\mathcal P$, and to be a diffeomorphism from each stratum onto the corresponding stratum.

\end{lemma}

\begin{proof}
For any point $(u_1,\dots,u_r) \in \R^r$ define real numbers $x_1,\dots,x_{r+1}$ in such a way that $x_{j+1}-x_j=u_j$ for all $j=1,\dots,r$ and $\sum_{j=1}^{r+1} x_j=0$; denote the resulting (injective linear) map $\mathbf u = (u_1,\dots,u_r) \mapsto \mathbf x = (x_1,\dots,x_{r+1})$ by $A$. Using the identification of Lemma \ref{lemma:polyroots} $b)$, define $\varphi: \R_\angle^r \to \overline \Delta$ by sending $(u_1,\dots,u_r)$ to the polynomial that has roots $x_1,\dots,x_{r+1}$. This map is clearly a homeomorphism, we only need to show that $\varphi$ maps strata onto strata diffeomorphically.
\par
For computational reasons, we extend $\varphi$ to a map $\Phi: \R^{r+1} \to \R^{r+1}$ by mapping the point $(u_1,\dots,u_r,\delta)$ to the coefficients of the monic polynomial with roots $x_1+\frac{1}{r+1}\delta,\dots,x_{r+1}+\frac{1}{r+1}\delta$. Then $\varphi(\mathbf u)$ can be identified with $\Phi(\mathbf u,0)$ and to prove our claim it is enough to show that $\Phi$ is a diffeomorphism from the strata of $\R^r_\angle \times \R$ onto their images. It is clear that $\Phi$ maps the strata of $\R^r_\angle \times \R$ homeomorphically onto their images, we only need to show that the rank of derivative at any point of an $s$-dimensional stratum is $s$.
\par
To do this, we write $\Phi$ as the composition
$$
\Phi=E \circ \tilde A,
$$
where the linear map $\tilde A$ is defined as
$$
\tilde A (\mathbf u,\delta) = -A\mathbf u - \frac{1}{r+1}(\delta,\dots,\delta)
$$
and $E : \R^{r+1} \to \R^{r+1}$ is the map whose $j$th coordinate function is the $j$th elementary symmetric function:
$$
E(x_1,\dots,x_{r+1})= \big( e_1(x_1,\dots,x_{r+1}),\dots,e_{r+1}(x_1,\dots,x_{r+1})\big).
$$
The map $\tilde A$ is a linear isomorphism, composing with it does not change the rank of the differential, therefore it is enough to show that the differential of $E$ has maximal rank when restricted to the strata of $\tilde A(\R^r_\angle\times \R)$. Note that these strata have the form
\begin{align*}
S(a_1,\dots,a_{s-1}) = \{ (x_1,\dots,x_{r+1}) : \quad & x_1 = \dots = x_{a_1-1} < \\
< & x_{a_1} = \dots = x_{a_2-1} <\\
< & \dots <  x_{a_{s-1}} = \dots =x_{r+1} \}
\end{align*}
for some $s\geq 1$ and an index set $1 < a_1 < a_2 < \dots < a_{s-1} \leq r+1$ (so that in particular $\dim S(a_1,\dots,a_{s-1}) =s$).
\par
The following claim is easily proved by induction on $r$:
\begin{claim}\label{claim:Jacobi}
The Jacobi matrix of $E$ at the point $\mathbf x = (x_1,\dots,x_{r+1})$ is 
$$
J(\mathbf x)=
\begin{bmatrix}
1&\dots&1&\dots&1\\
e_1(x_2,\dots,x_{r+1}) & \dots &  e_1(\dots,x_{j-1},x_{j+1},\dots)  & \dots & e_1(x_1,\dots,x_{r}) \\
e_2(x_2,\dots,x_{r+1}) & \dots &  e_2(\dots,x_{j-1},x_{j+1},\dots)  & \dots & e_2(x_1,\dots,x_{r}) \\
\vdots & & \vdots &  & \vdots\\
e_r(x_2,\dots,x_{r+1}) & \dots &  e_r(\dots,x_{j-1},x_{j+1},\dots)  & \dots & e_r(x_1,\dots,x_{r}) \\
\end{bmatrix}.
$$
Its determinant is $\prod_{1\leq i<j \leq r+1} (x_i-x_j).$
\end{claim}
In particular, this matrix is nondegenerate if the $x_j$ are pairwise different. To estimate the rank of the differential of $E$ restricted to $S(a_1,\dots,a_{s-1})$ at a point $\mathbf x \in S(a_1,\dots,a_{s-1})$, notice that the columns of $J(\mathbf x)$ with indices $i$ and $j$ coincide if $x_i=x_j$, therefore $\rk dE(\mathbf x) \leq s$. On the other hand, the columns with indices $1$, $a_1$, $\dots$, $a_{s-1}$ are linearly independent -- the minor formed by the first $s$ rows of these columns can be calculated to be nonzero in the same way as in Claim \ref{claim:Jacobi}. Hence $\rk dE(\mathbf x) =s$ and we get that the kernel of $dE(\mathbf x)$ is spanned by the vectors $\mathbf u_i-\mathbf u_j$ for those $i$ and $j$ for which $x_i=x_j$, where $\mathbf u_l$ denotes the $l$th unit coordinate vector. All these vectors are orthogonal to $S(a_1,\dots,a_{s-1})$, consequently their span is also transverse to $S(a_1,\dots,a_{s-1})$ and hence the restriction of $E$ to $S(a_1,\dots,a_{s-1})$ has full rank because its differential is the restriction of the differential $dE$ to the tangent space of $S(a_1,\dots,a_{s-1})$.
\end{proof}

\begin{figure}
\centering
\resizebox{10cm}{!}{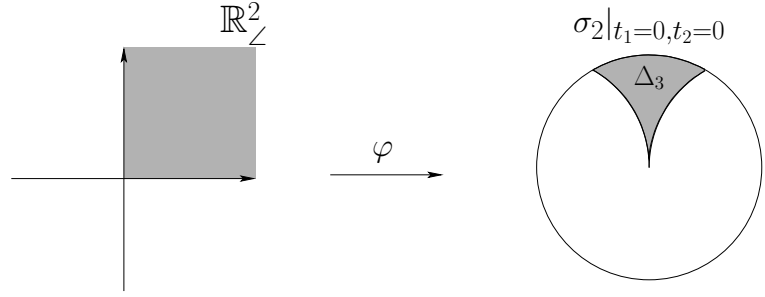}
\caption{The homeomorphism $\varphi$ for $r=2$.}
\label{fig:cusp}
\end{figure}

\par
We want to show that $\overline\Delta$ has a (homotopically) canonical parallelization. Clearly any homeomorphism $\varphi$ that is a diffeomorphism on the strata does give a parallelization induced from $\R^{r}_\angle$. But there are many possible choices of $\varphi$ (even up to isotopy), and to obtain a parallelization along the whole $\overline\Delta$ that does not depend on the choice of local coordinates we need to show that there is a canonical (unique up to isotopy) such choice. For this purpose it is enough to show that there is a canonical ordering of the $1$-dimensional edges of $\overline \Delta$, because if $\varphi$ respects the ordering of the edges (and it can be chosen to do so), then it is isotopically unique. We shall consider $\sigma_r$ as a prim map, namely the projection of the immersion $(t_1,\dots,t_{2r-1},x) \mapsto (\sigma_r(t_1,\dots,t_{2r-1},x),x) \in \R^{2r+2}$. This gives an ordering of the preimages of multiple points: for an $(r+1)$-tuple point with preimages $a_1,\dots,a_{r+1}$ we assume that the indexing is such that $i<j$ if and only if $x(a_i)<x(a_j)$. If we choose a sequence $(q_n) \in \Delta$ converging to a point $\overline q \in \overline \Delta \setminus 0$ that lies on a $1$-dimensional edge, then the $\sigma_r$-preimages of the $(r+1)$-tuple points $\sigma_r^{-1}(q_n)=\{ a_1^{(n)}, \dots, a_{r+1}^{(n)} \}$ degenerate in the limit in the sense that there is an integer $s$, $1 \leq s \leq r$ and there are two different points $\underline a$ and $\overline a$ with $x(\underline a)<x(\overline a)$ such that $\underset{n \to \infty}{\lim} a_i^{(n)} = \underline a$ for $i=1,\dots,s$ while $\underset{n \to \infty}{\lim} a_i^{(n)} = \overline a$ for $i=s+1,\dots,r+1$. Hence to each edge we can associate an integer $s$, $1\leq s \leq r$, and thus we obtain an ordering of the edges of $\overline \Delta$. We choose the map $\varphi$ in such a way that it respects the ordering of the edges.
\par
Note that any automorphism of $\sigma_r$ as a prim map keeps the ordering of the edges of $\overline \Delta$. We call a pair $(\alpha,\beta)$ of germs of diffeomorphisms $\alpha: (\R^{2r},0) \to (\R^{2r},0)$ and $\beta: (\R^{2r+1},0) \to (\R^{2r+1},0)$ a prim automorphism of the prim germ $\sigma_r = \pi \circ \hat\sigma_r: (\R^{2r},0) \looparrowright (\R^{2r+2},0) \to (\R^{2r+1},0)$ if in addition to being an automorphism of $\sigma_r$ (that is, $\sigma_r \circ \alpha = \beta \circ \sigma_r$) it preserves the selected orientation of $\ker \mathrm{d}\sigma_r(0)$, i.e. the partial derivative of the $2r$th coordinate function of $\alpha$ with respect to $x$ is positive (recall that $\frac{\partial \sigma_r}{\partial x}(0) =0$).
\par
Let now $g: M^n \looparrowright \R^{n+2}$ be an immersion, where $M$ is a compact oriented $n$-dimensional manifold. Let $f$ be the prim map $f=\pi\circ g$, with $\pi: \R^{n+2} \to \R^{n+1}$ the projection that omits the last coordinate $x_{n+2}$ in $\R^{n+2}$. Suppose that $f$ is a $\Sigma^{1_r}$-map, that is, it has no $\Sigma^{1_j}$ points for $j > r$. The set of $\Sigma^{1_r}$-points of $f$ will be denoted by $\Sigma^{1_r}(f)$, or $\Sigma$ for brevity. It is a submanifold of $M$ of dimension $n-2r$. Its image $f(\Sigma^{1_r}(f))$ (denoted by $\tilde \Sigma$) is an immersed submanifold in $\R^{n+1}$ and has codimension $2r+1$. We claim that $\tilde \Sigma$ has a natural normal framing, unique up to homotopy. To simplify the presentation of the proof, we will suppose that $\tilde \Sigma$ is embedded into $\R^{n+1}$.
\par
There is an embedding $\theta: \Sigma \times (-\varepsilon, \varepsilon) \hookrightarrow \R^{n+1}$ (for some small positive number $\varepsilon$), unique up to isotopy, such that
\begin{itemize}
\item $\theta(x,0)=f(x)$ for all $x\in \Sigma$,
\item $\theta(x,t)$ is an $(r+1)$-tuple point of $f$ for all $0<t<\varepsilon$.
%% \item $\theta(x,t) - \theta(x,0)$ is orthogonal (with respect to the standard scalar product on $\R^{n+1}$) to the tangent space $T_{\theta(x,0)}\tilde \Sigma$.
\end{itemize}
Choose any $t^* \in (0, \varepsilon)$ and denote by $\Sigma^*$ the image set $\theta(\Sigma\times\{t^*\})$. Clearly it is enough to give a canonical normal framing of $\Sigma^*$. Let $\Delta_{r+1}(f)$ denote the set of $(r+1)$-tuple points of $f$ that lie in a tubular neighbourhood of $\tilde \Sigma$ and let $\overline{\Delta_{r+1}(f)}$ be its relative closure. There is a fibration $\overline{\Delta_{r+1}(f)} \to \Sigma^*$ with fiber $\overline \Delta$.
\par
The normal bundle of $\Sigma^*$ in $\R^{n+1}$ is the sum of its normal bundle in $\Delta_{r+1}(f)$ and the restriction of the normal bundle of $\Delta_{r+1}(f)$ in $\R^{n+1}$ to $\Sigma^*$. The latter bundle is trivial, because it is the sum of the trivial normal line bundles of the $(r+1)$ non-singular branches of $f$ that intersect at the points of $\Delta_{r+1}(f)$ (and these branches have a canonical ordering by the last coordinate of $g$). The former bundle is trivial because there exists a canonical parallelization of $\overline \Delta$ that is (homotopically) invariant under the prim automorphisms of $\sigma_r$.
\par
In general, when $\tilde \Sigma$ is not embedded, we need to consider only the preimages of multiple points that are close to the singular points in the source manifold to make the same argument work.
\par
Alternatively, one can notice that the structure group of the bundle $\overline{\Delta_{r+1}(f)} \overset{\overline\Delta}{\to} \Sigma^*$ can be reduced to the maximal compact subgroup of the group of automorphisms of $\sigma_r$ that also respect the orientation of the kernel of $d\sigma_r$, and that group is trivial \cite{RSz}. Consequently the normal bundle of $\Sigma^*$ in $\Delta_{r+1}(f)$ has a homotopically unique trivialization. We have seen above that the normal bundle of $\Delta_{r+1}(f)$ admits a canonical trivialization as well. Therefore we get a (homotopically unique) normal framing of $\Sigma^*$.

\section*{Appendix 3: Proof of the Compression Theorem}

Throughout the proof we will assume that all the vector fields $v_j$ have unit length and are pairwise orthogonal. To see that this can be achieved, assume that in the process of the proof an isotopy $\wt\Phi_t$ has already been constructed but it does not satisfy our assumption; we will now correct it to make it preserve the lengths and the pairwise orthogonality of the fields $v_j$. Let $A(t,p)$ for all $t\in \R$ and $p\in i(M)$ be the (unique) linear transformation of the linear space $N(t,p) = \langle d\wt \Phi_t (v_1(p)), \dots,d\wt\Phi_t (v_k(p)) \rangle$ that sends the base $d\wt \Phi_t (v_1(p))$, $\dots$, $d\wt\Phi_t (v_k(p))$ to the orthonormed base obtained from it by the Gram-Schmidt orthonormalization process. There exists an orthonorming isotopy $\mathscr O_t$ of $\R^n \times \R^k$ that fixes $i(M)$ and the complement of a tubular neighbourhood of $i(M)$ pointwise such that for all $p\in i(M)$ it maps a small disk of the affine space $p + N(0,p)$ into the same affine space by the affine transformation $d\wt \Phi_t^{-1}|_{\Phi_t(p)} \circ A(t,p) \circ d\wt\Phi_t|_p$. Defining $\Phi_t = \wt\Phi_t \circ \mathscr O_t$ yields the claimed isotopy. It is easy to check that this construction can be performed in such a way as to respect the extra requirements of statements $b)$ and $d)$ of Theorem \ref{thm:multicompression} -- the details are mentioned at the corresponding points of the proof. Repeating this process every time a new isotopy is introduced ensures that the fields $v_j$ stay pairwise orthogonal throughout the construction.

We first address the {\bf case $k=1$}, when we have an embedding $i : M \hookrightarrow \R^n \times \R^1$. For brevity the unit normal vector field $v_1$ will be denoted by $v$, the vector $\mathbf e_1$ will be denoted by $\up$ and $\down$ will be the vector $-\mathbf e_1$. We think of the direction of $\mathbf 0 \times \R^1$ as vertical, and those in $\R^n \times 0$ as horizontal. We may and will assume that $v$ is perpendicular to $i(M)$ (but it will not remain so during the constructed isotopy).

{\bf Case $k=1$, $n>m$.}
\par
{\bf Step $1$:} construct a diffeomorphism $\Psi \in \Diff(\R^n \times \R^1)$ in order to obtain the reparametrization $i' = \Psi \circ i$ such that
\begin{itemize}
\item the image $d\Psi(v)$ of $v$ under the differential of $\Psi$ -- we denote it by $w$ -- is orthogonal to $i'(M)$ and \emph{grounded}, that is, $w$ is never directed~parallel to $\down$;
\item $\Psi$ is the identity outside a compact subset of $\R^n \times \R^1$.
\end{itemize}
We perform the construction by taking the flow $\Psi_t$ of a vector field $\bm \theta$ on $\R^n \times \R^1$ and setting $\Psi=\Psi_\varepsilon$ for a suitably chosen $\varepsilon>0$; we start by constructing $\bm \theta$. Since $m<n$, the image of $v$ in the unit sphere $S^n$ has measure zero. Consequently, for almost every horizontal direction $\mathbf h \in S^{n-1} \hookrightarrow S^n$ the arc $\{  \cos \lambda \cdot \down + \sin \lambda \cdot \mathbf h : \lambda \in [0,\pi] \}$ intersects the image of $v$ in a null set according to the $1$-dimensional Lebesgue measure. Taking such an $\mathbf h$, let $R$ denote the rotation of $\R^n\times\R^1$ that turns $\mathbf h$ into $\down$ and is the identity on the orthogonal complement of the $2$-dimensional linear subspace spanned by $\down$ and $\mathbf h$. Let the vector field $\bm \theta$ be equal to the infinitesimal generator of $R$ in a tubular neighbourhood $V$ of $i(M)$ and extend it arbitrarily to a smooth vector field on $\R^n\times\R^1$ that vanishes outside a compact set. For all sufficiently small $\varepsilon >0$ the time-$\varepsilon$ flow of the vector field $\bm \theta$ will keep the points of $i(M)$ within the tubular neighbourhood $V$, hence the vector field $v$ will be moved to $R^\varepsilon v$ and for almost all such $\varepsilon$ the rotated vector field $R^\varepsilon v$ will miss $\down$. Note that the diffeomorphism $\Psi$ preserves the orthogonality of $v$ to the tangent space of $M$ (i.e. $w \perp i'(M)$).
\par
{\bf Step $2$:} denote by $b$ the inner bisector vector field (along $i'(M)$) of $\up$ and $w$. Note that
\begin{enumerate}[a)]
\item $b$ nowhere belongs to the tangent bundle $T(i'(M))$;
\item the scalar product $\langle b,\up \rangle$ is everywhere positive.
\end{enumerate}
In this step we construct an isotopy $\Xi_t \in \Diff(\R^n \times \R^1)$, $t \in \R$ such that
\begin{enumerate}[(i)]
\item $\Xi_0$ is the identity;
\item for some $t_0>0$ the image $d\Xi_{t_0}(b)$ of $b$ under $d\Xi_{t_0}$ is $\up$;
\item $\Xi_{t_0+t}(x)=\Xi_{t_0}(x)+t\cdot\up$ for all $t>0$ and for all $x\in \R^n\times\R^1$.
\end{enumerate}
We proceed like in \cite{compression}, by extending the vector field $b$ given on $i'(M)$ to a vector field $\mathbf b$ on $\R^n\times \R^1$ in such a way that $\mathbf b$ also satisfies $\langle \mathbf b, \up \rangle > \delta > 0$ everywhere for some $\delta>0$ and $\mathbf b = \up$ holds outside a tubular neighbourhood $V'$ of $i'(M)$. Let $\wt \Xi_t$ be the $1$-parameter family of diffeomorphisms generated by the vector field $\mathbf b$; note that $\wt \Xi_t$ satisfies properties (i) and (ii), while property (iii) holds for $x\in i'(M)$ but not everywhere else. In order to define $\Xi_t$ that also satisfies property (iii) we choose a function $\varphi: [0,\infty) \to [0,1]$ such that
\begin{itemize}
\item $\varphi(t)=1$ for all $t$ for which $\wt \Xi_t(i'(M)) \cap V' \neq \emptyset$, and
\item $\varphi(t) = 0$ for all $t\geq t_0$ for some $t_0>0$.
\end{itemize}
We define $\Xi_t$ to be the $1$-parameter family of diffeomorphisms generated by the vector field $\mathbf{\hat b}_t = \varphi(t) \cdot \mathbf b + (1-\varphi(t))\cdot \up$.

Since in a neighbourhood of $\wt\Xi_t(i'(M))$ the vector fields $\mathbf b$ and $\mathbf{\hat b}_t$ coincide, so do the flows $\wt\Xi_t$ and $\Xi_t$ in a neighbourhood of $i'(M)$. This ensures that $\Xi_t$ still satisfies properties (i) and (ii), and property (iii) holds by the definition of $\mathbf{\hat b}_t$.
\par
{\bf Step $3$:} construct another isotopy $\Theta_t \in \Diff(\R^n\times\R^1)$, $t \in [0,1]$ (with $\Theta_0$ the identity) that keeps $i'(M)$ fixed pointwise and turns the vector $w$ into $b$. First, we set $w_\tau$, $\tau\in [0,1]$ to be the vector field on $i'(M)$ obtained by smoothly rotating $w=w_0$ into $b=w_1$ along the shortest arc connecting them, staying constant in neighbourhoods of $\tau=0$ and of $\tau=1$, respectively. Since $w$ is orthogonal to $i'(M)$ and $b$ forms an acute angle with $w$, the vector field $w_\tau$ is not tangent to $i'(M)$ for any $\tau\in [0,1]$. Consequently the map
\begin{align*}  %% \tag{*}\label{eq:embedding}
I_\tau: M \times [0,1] & \to \R^{n+1} \\
(p,s) & \mapsto i'(p)+\varepsilon_\tau \cdot s \cdot w_\tau\left(i'(p)\right)
\end{align*}
is an embedding for every $\tau$ for some sufficiently small $\varepsilon_\tau > 0$.
\par
By compactness of the interval $[0,1]$ there is a common $\varepsilon>0$ such that $I_\tau$ is an embedding for all $\tau$ with the choice $\varepsilon_\tau= \varepsilon$. Then $I_\tau$ is an isotopy of embeddings of the cylinder $M \times [0,1]$ as $\tau$ changes from $0$ to $1$. This isotopy can be extended to an isotopy of $\R^n\times\R^1$, which we denote by $\Theta_t$; since $w_\tau$ stayed constant in a neighbourhood of $\tau=1$, we may assume that $\Theta_t$ also stays constant in a neighbourhood of $t=1$. The isotopy $\Theta_t$ can then be extended to all times $t \in [0,\infty)$ by putting $\Theta_t=\Theta_1$ for $t>1$.
\par
{\bf Step $4$:} we compose the isotopies constructed in the previous steps:
$$\Phi_t = \Psi^{-1} \circ \Xi_t \circ \Theta_t \circ \Psi.$$
We claim that for $t$ sufficiently big this isotopy will turn the vector field $v$ into the constant vector field $\up$. Indeed, $\Psi$ sends $v$ to $w$; for $t$ sufficiently big $\Theta_t$ sends $w$ into $b$; and $\Xi_t$ sends $b$ into $\up$. Additionally, $\Xi_t \circ \Theta_t |_{i'(M)} = \Xi_t|_{i'(M)}$ moves the image of $M$ at a constant nonzero velocity for $t>t_0$. In finite time it leaves the compact set where $\Psi$ differs from the identity map. Therefore $\Psi^{-1}$ is the identity on the image of $\Xi_t \circ \Theta_t \circ \Psi \circ i$ for $t$ big enough. That is, the image of $v$ under $\Phi_t$ is eventually $\up$.
\par
We need to show that the map $(p,t) \mapsto \Phi_t(i(p))$ is an immersion of $M \times [0,\infty)$ into $\R^n \times \R^1$. Let $p\in M$ be an arbitrary point and consider the image of the ray $p\times [0,\infty)$: it is the curve $\Psi^{-1}\left(\Xi_t\left(\Theta_t \left(\Psi\left(i(p)\right)\right)\right)\right) = \Psi^{-1}\left(\Xi_t\left(\Theta_t \left(i'(p)\right)\right)\right)$. The isotopy $\Theta_t$ keeps $i'(p)$ fixed, and the curve $\Xi_t(i'(p))$ ($= \wt\Xi_t(i'(p))$) is the trajectory of the point $i'(p)$ under the (time-independent) flow of the vector field $\mathbf b$. The derivative of this trajectory at the starting point $i'(p)$ is $b(p)$, a vector linearly independent of $T_{i'(p)}i'(M)$, hence the same independence condition holds for the images under $\Xi_t$ at any time $t>0$. A final composition with the diffeomorphism $\Psi^{-1}$ keeps the condition intact.
\par
{\bf Case $k=1$, $n=m$.} Since $M$ has no closed components, there is an $(n-1)$-dimensional compact subcomplex $K$ in $M$ such that for any neighbourhood $U$ of $K$ there is an isotopy $\alpha^U_\tau$, $\tau \in [0, 1]$ of $M$ that deforms $M$ into $U$, i.e. $\alpha^U_0 = id_M$ and $\alpha^U_1(M) \subset U$. The neighbourhood $U$ will be chosen later and $\alpha_\tau$ will denote the isotopy corresponding to that choice of~$U$.
\par
We consider the compressing procedure (turning $v$ to $\up$) of the case ${n>m}$ on the subcomplex $K$. The result is an isotopy $\hat \Phi_t$ such that for some $t_0$ and for all $t\geq t_0$
\begin{itemize}
\item $d\hat\Phi_t(v) = \up$ holds on $K$;
\item $\hat\Phi_t(p) = \hat\Phi_{t_0}(p)+(t-t_0) \cdot \up$ holds for all $p\in K$;
\item for a neighbourhood $\mathcal N$ of $K$ in $M$ the restriction of the isotopy to the image $i(\mathcal N)$ is an immersion of the infinite cylinder $\mathcal N \times [0,\infty)$.
\end{itemize}
Since on $K$ the image $d\hat\Phi_{t_0}(v) = \up$ does not belong to the tangent space $T\hat\Phi_{t_0}(i(M))$, there is neighbourhood $U\subset \mathcal N$ of $K$ in $M$ such that in the closure of $U$ the shortest arc connecting the image $d\hat\Phi_{t_0}(v|_{i(\overline U)})$ with $\up$ does not intersect $T\hat\Phi_{t_0}(i(M))$. This neighbourhood $U$ will be the one defining $\alpha_\tau$. We also consider the following extension $\overline \alpha_\tau$ of $\alpha_\tau$ to the cylinder $M \times [0,1]$ in the ambient space $\R^n \times \R^1$:
\begin{align*}
\overline\alpha_\tau(p,s)& = i(\alpha_\tau(p))+ \varepsilon \cdot s \cdot v(i(\alpha_\tau(p))) \quad\text{for }p\in M,\ s\in [0,1]
\end{align*}
for an $\varepsilon>0$ small enough so that $\overline \alpha_0$ (and consequently every $\overline \alpha_\tau$) is an embedding, and extend it to an isotopy $\wt \alpha_\tau$ of $\R^n \times \R^1$. It can be supposed that $\wt \alpha_\tau$ is the identity outside a compact set $B \subset \R^n\times \R^1$. By increasing $t_0$ if needed we may also assume that for all $t\geq t_0$ we have $\hat\Phi_t(i(M)) \cap B = \emptyset$.
\par
In the neighbourhood $U$ we can again repeat the argument of Step $3$ to obtain an isotopy $\Lambda_t$, $t\geq 0$ of the identity of $\R^n\times \R^1$ that fixes $i(M)$ pointwise, turns $v$ into $d\hat\Phi_{t_0}^{-1}(\up)|_{i(\overline U)}$ and stays constant $\Lambda_t = \Lambda_{t_0}$ for $t\geq t_0$. We can now define the isotopy
$$
\Phi_t = \wt\alpha_1^{-1} \circ \hat\Phi_t \circ \Lambda_t \circ \wt\alpha_1.
$$
This isotopy will send $v$ to $\up$ for $t \geq t_0$ (for $t \geq t_0$ the image $\hat\Phi_t(i(M))$ is disjoint from $B$, hence $\wt\alpha_1^{-1}$ is the identity near $\hat\Phi_t(i(M))$), and on $i(M)$ we have that $\Phi_t = \wt\alpha_1^{-1} \circ \hat\Phi_t \circ \Lambda_t \circ \wt\alpha_1 = \wt\alpha_1^{-1} \circ \hat\Phi_t \circ \wt\alpha_1$ is a reparametrization of $\hat\Phi_t|_{i(U)}$, hence the restriction of the obtained isotopy to $i(M)$ is an immersion of the cylinder $M \times [0,\infty)$ as required.
\par
{\bf Case $k>1$.} Here we perform the construction in two steps, first straightening out the vector field $v_1$ as detailed above both when $n=m$ and when $n>m$, and then utilizing different corollaries of the original Rourke-Sanderson compression theorem appropriate for the cases $n=m$ and $n>m$ to straighten the rest of the normal vector fields. In the first step we obtain an isotopy $\Phi^{(0)}_t$ such that for some $t_0$ and all $t>t_0$ we have $\Phi^{(0)}_t(p)=\Phi^{(0)}_{t_0}(p)+(t-t_0) \mathbf e_1$ for all $p\in \R^n\times\R^k$ and $d\Phi^{(0)}_t(v_1)=\mathbf e_1$, while the restriction of the isotopy to $i(M)$ is an immersion of the cylinder $M\times [0,\infty)$. Then we apply \cite[Multi-compression Theorem 4.5]{compression} when $n>m$ and \cite[Addendum (v) to Multi-compression Theorem 4.5]{compression} when $n=m$ to the composition $\Phi^{(0)}_{t_0} \circ i$ and obtain an isotopy $\Phi^{(1)}_t$, $t\in [0,t_1]$, such that it straightens the images of all the vector fields $v_1$, $\dots$, $v_k$. By construction the isotopy $\Phi^{(1)}_t$ is a lift of a $C^0$-small regular homotopy $\check\Phi_t$ of the projection of $\Phi^{(0)}_{t_0}(i(M))$ parallel to $\mathbf e_1$, hence we can choose this lift to be locally of the form $\check\Phi_t \times id_{\langle \mathbf e_1\rangle}$ and consequently preserve the first coordinate function as well as the $v_1$ direction.
\par
By reparametrizing time we may assume that $\Phi^{(1)}_t$ can be extended smoothly to all $t\in \R$ as the identity for all $t<0$ and as a time-independent diffeomorphism for all $t>t_0$, and by multiplying the time-dependent generating vector field $\frac{\partial \Phi^{(1)}_t}{\partial t}$ by an appropriate bump function we may also assume that $\Phi^{(1)}_t$ only moves points within a compact subset of $\R^n\times\R^k$.
\par
We define
\begin{equation}\label{eq:defPhi}\tag{$\diamondsuit$}
\Phi_t(p)=\begin{cases}\Phi^{(0)}_t(p) &\text{ if } t\leq t_0,\\\Phi^{(1)}_{t-t_0}(\Phi^{(0)}_{t_0}(p))+(t-t_0) \mathbf e_1 &\text{ if } t\geq t_0.\end{cases}
\end{equation}
This is an isotopy that straightens all the vector fields $v_1$, $\dots$, $v_k$, we only need to check that its restriction to $i(M)$ is an immersion of the cylinder $M\times [0,\infty)$. For times $t\leq t_0$ this is already proven above in the $k=1$ case. For $t\geq t_0$ we have
$$
\frac{\partial \Phi_t(p)}{\partial t} = \frac{\partial \Phi^{(1)}_{t-t_0}(\Phi^{(0)}_{t_0}(p))}{\partial t} + \mathbf e_1,
$$
and this is not in the tangent space of $\Phi_t(i(M))$ exactly if the image of $\frac{\partial \Phi^{(1)}_{t-t_0}(\Phi^{(0)}_{t_0}(p))}{\partial t}$ avoids $-\mathbf e_1 + d\Phi_t(di(TM))$. This latter condition can be achieved by a linear reparametrization of time in $\Phi^{(1)}_t$: we can slow down $\Phi^{(1)}_t$ by a sufficiently small factor $\varepsilon>0$ so that the length of the derivative $\frac{\partial \Phi^{(1)}_{\varepsilon t}}{\partial t}$ becomes always less than the minimum of the distance between $\mathbf e_1$ and $\im d\Phi_t$ on $\Phi^{(0)}_{t_0}(i(M))$ -- this latter is positive since the projection of $\Phi^{(1)}_t$ parallel to the $\mathbf e_1$ direction is a regular homotopy of the projection of $\Phi^{(0)}_{t_0}(i(M))$.
\par
For part $b)$ of Theorem~\ref{thm:multicompression} we proceed in the same way: first we straighten the vector field $v_{\hat\jmath}$ for some $\hat\jmath\not\in I$ (if $I=\{1,\dots,k\}$, then we are already done). Next we apply \cite[Addendum (vi) to Multi-compression Theorem 4.5]{compression} to the projection of $\Phi^{(0)}_{t_0}(i(M))$ parallel to the linear span $\langle \mathbf e_h: h\in I \cup \{ \hat\jmath \}\rangle$ and obtain a regular homotopy of this projection. We choose an isotopy lift $\Phi^{(1)}_t$ of the obtained regular homotopy that keeps coordinates with indices in $I \cup \{\hat\jmath \}$ fixed and combine it with the constant speed shift in the $\mathbf e_{\hat\jmath}$ direction as given by \eqref{eq:defPhi}. This yields an isotopy $\Phi_t$ that straightens all the vector fields $v_1$, $\dots$, $v_k$ and keeps the coordinates that belong to $I$ fixed. Therefore we only need to perform the straightening of the vectorfield $v_{\hat\jmath}$ in a manner that satisfies the requirements of part $b)$ of the theorem. In order to keep unchanged the coordinates that belong to $I$, all the vector fields that generate the isotopies forming the final isotopy $\Phi_t$ will be tangent to the orthogonal complement of the linear span of $\mathbf e_j$, the linear space $W=W_I = \langle \mathbf e_i : i\in I \rangle^\perp$:
\begin{itemize}
\item In Step $1$, the vector field $v_j$ lies in the $(n+k-\vert I\vert-1)$-dimensional unit sphere of the linear space $W$, and we can make it miss the $\down$ direction by a small move parallel to $W$.
\item Step $2$ can be performed to yield a vector field $\mathbf b$ parallel to $W$ since $v_j$ and its target state $\mathbf e_j$ are both parallel to $W$.
\item Step $3$ can be performed so that the isotopy $\Theta_t$ preserves parallel translates of $W$ since both $v_{\hat\jmath}$ and $b$ are parallel to $W$.
\item The orthonorming isotopy $\mathscr O_t$ mentioned in the beginning of Appendix $3$ is constructed as an extension of an isotopy that acts affinely on normal disks of $i(M)$. Each of these affine transformations fixes a point in $i(M)$ and has a linear part that is a composition of the differential of $\Phi_t$, its inverse map and the maps $A(t,p)$; all of these linear maps preserve $W$ if the Gram-Schmidt orthogonalization process is run on the $v_j$ with $j\in I$ first. Hence $\mathscr O_t$ is an extension of an isotopy that preserves translates of $W$ and can therefore be chosen to preserve translates of $W$ itself. Consequently the corrected isotopy $\Phi_t \circ \mathscr O_t$ will preserve translates of $W$ as well.
\end{itemize}
If, in the course of the proof, instead of $di(T_pM)$ we use its orthogonal projection onto $W$, $(di(T_pM) + W^\perp) \cap (W + i(p))$, then not only will the map $(p,t) \mapsto \Phi_t(i(p))$ be an immersion of $M\times [0,\infty)$, but the map
$$\left(p,(x_j)_{j\in I},t\right) \mapsto \Phi_t(i(p))+\sum_{j\in I}x_j \cdot v_j
$$
will be an immersion of $M \times D^{\vert I \vert}_\varepsilon \times [0,\infty)$ for some sufficiently small positive $\varepsilon$.
\par
The {\bf relative version} of Theorem~\ref{thm:multicompression} only requires substantial adaption of Steps $1$ and $2$ of the proof -- in Step $3$, we only have to additionally require that $\Theta_t$ is the identity on $L\times \R^k$ at all times, and in the construction of $\mathscr O_t$ (see the beginning of Appendix $3$) we have to require that $\mathscr O_t$ should be the identity on a neighbourhood of $L \times \R^k$. We shall again consider only the field $v_1$, which we denote by~$v$. If the vector field $v$ is grounded ($v \neq \down$ everywhere) then Step $1$ can be skipped (setting $\Psi=id_{\R^n\times\R^1}$) and no further change in the proof is needed. If the subset $\{ p \in M : v(i(p)) = \down \}$ is non-empty, then we need to modify the definition of the vector field $\bm \theta$ in Step $1$ since its flow does not necessarily preserve $v$ on $L\times \R^1$. We choose $\bm \theta$ to satisfy the following for some positive $\delta < \pi/6$ ($\angle$ stands for angle):
\begin{itemize}
\item $\bm \theta$ is the infinitesimal generator of the rotation $R$ (chosen as in the ${m<n}$ case) on a neighbourhood of the set
$$\mathscr U=\{ p \in M : \angle (v(i(p)), \up) \geq \delta \};$$
\item $\bm \theta = \mathbf 0$ outside a compact set;
\item $\bm \theta = \mathbf 0$ in a neighbourhood of $L\times \R^1$.
\end{itemize}
While this $\bm \theta$ no longer preserves orthogonality, taking a sufficiently small $\varepsilon$ in the construction of $\Psi=\Psi_\varepsilon$ will make $\Psi$
\begin{itemize}
\item preserve orthogonality on $i(\mathscr U)$;
\item change direction of all vectors by less than $\delta$;
\item act as the identity in a neighbourhood of $L \times \R^1$.
\end{itemize}
Hence the image of $v$ under $\Psi$ (again denoted by $w$) will still miss $\down$ -- at points in $\mathscr U$ the diffeomorphism $\Psi$ is the fixed rotation $R^\varepsilon$ and at points outside $\mathscr U$ a rotation of an angle less than $\delta$ cannot turn them into the lower hemisphere $\{ z \in S^n : \langle z, \up \rangle <0 \}$. On $L\times \R^1$ the vector fields $v$ and $w$ coincide. We now define the vector field $b$ to be the inner bisector of $\up$ and $w$ as in Step $2$; once its property $a)$ is verified, the rest of the proof will proceed without change (property $b)$ holds trivially).
\par
To check condition $a)$ of Step $2$, first note that for any vector $w \neq \down$ the new vector $b$ forms an acute angle with $w$, hence the condition is satisfied on $\mathscr U$, where $w$ is orthogonal to $M$. Outside $\mathscr U$ the angle $\angle(w,\up)$ is less than $2\delta$ since $\Psi$ changed all angles by at most $\delta$. At these points the angle $\angle (b,w)$ is at most $\delta$, hence $b(i'(p))$ and $v(i(p))$ form an angle of at most $\delta+\delta = 2\delta$ for all $p \in M \setminus \mathscr U$. Since the tangent space of $i(M)$ was also rotated by $\Psi$ by an angle less than $\delta$, the vector field $b$ and $Ti'(M)$ form a strictly positive angle as required by condition $a)$ since we chose $\delta<\pi/6$.

\newpage
\small

\noindent
Authors addresses:

\noindent
Csaba Nagy\\
Department of Analysis\\
E\"otv\"os Lor\'and University (ELTE)\\
Budapest, P\'azm\'any P. s\'et\'any I/C\\
H-1117 Hungary\\
e-mail: csaba224@freemail.hu\\
\\
Andr\'as Sz\H{u}cs\\
Department of Analysis\\
E\"otv\"os Lor\'and University (ELTE)\\
Budapest, P\'azm\'any P. s\'et\'any I/C\\
H-1117 Hungary\\
e-mail: szucs@cs.elte.hu\\
\\
Tam\'as Terpai\\
A. R\'enyi Mathematical Institute\\
of the Hungarian Academy of Sciences\\
Budapest, Re\'altanoda u. 13--15\\
H-1053 Hungary\\
e-mail: terpai@renyi.mta.hu

\end{document}